\documentclass[leqno,english]{amsart}
\usepackage{amsfonts,amssymb,amsmath,amsgen,amsthm}
\usepackage{hyperref}
\usepackage{color}
\newcommand{\msc}[2][2000]{%
  \let\@oldtitle\@title%
  \gdef\@title{\@oldtitle\footnotetext{#1 \emph{Mathematics subject
        classification.} #2}}%
}

\theoremstyle{plain}
\newtheorem{theorem}{Theorem} [section]

\newtheorem{lemma}[theorem]{Lemma}

\newtheorem{proposition}[theorem]{Proposition}
\newtheorem{hyp}[theorem]{Assumption}
\theoremstyle{remark}
\newtheorem{remark}[theorem]{Remark}

\def\dis
{\displaystyle}

\def\O{\mathcal O}
\def\C{{\mathbb C}}
\def\R{{\mathbb R}}
\def\N{{\mathbb N}}
\def\O{\mathcal O}

\def\({\left(}
\def\){\right)}
\def\<{\left\langle}
\def\>{\right\rangle}
\def\le{\leqslant}
\def\ge{\geqslant}

\def\Tend#1#2{\mathop{\longrightarrow}\limits_{#1\rightarrow#2}}

\def\d{{\partial}}
\def\eps{\varepsilon}

\def\si{{\sigma}}

\DeclareMathOperator{\RE}{Re}
\DeclareMathOperator{\IM}{Im}

\numberwithin{equation}{section}

\begin{document}

\title[Coherent states in logNLS]{Propagation of coherent states in
  the logarithmic Schr\"odinger equation} 
\author[R. Carles]{R\'emi Carles}
\address{CNRS\\ IRMAR - UMR
  6625\\ F-35000 Rennes, France}
\email{Remi.Carles@math.cnrs.fr}
\author[F. Dong]{Fangyuan Dong}
\address{School of Mathematics and Physics\\
University of Science and Technology Beijing\\
Xueyuan Road 30, Haidian\\ Beijing 100083, China}
\email{math.dongfy@gmail.com}

\begin{abstract}
We consider the logarithmic Schr\"odinger equation in a semiclassical
scaling, in the presence of a smooth, at most quadratic, external
potential. For initial data under the form of a single coherent state, we
identify the notion of criticality as far as the nonlinear coupling
constant is concerned, in the semiclassical limit. In the critical
case, we prove a general error estimate, and improve it in the case of
initial Gaussian profiles. In this critical case, when the initial
datum is the sum of two Gaussian coherent states with different
centers in phase space, we prove a nonlinear superposition principle. 
\end{abstract}
\thanks{The second author is on leave from the University
of Science and Technology in Beijing, thanks to
some funding from the Chinese Scholarship Council. A CC-BY public
copyright license has been applied by the authors to the present
document and will be applied to all subsequent versions up to the
Author Accepted Manuscript arising from this submission. }  
\maketitle

\section{Introduction}
\label{ref:intro}

We consider the logarithmic Schr\"odinger equation in the
semiclassical r\'egime $\eps\in( 0,1]$, 
\begin{equation}\label{eq:logNLS}
    i \varepsilon\partial_t \psi^\varepsilon+\frac{\varepsilon^2}{2}
    \Delta \psi^\varepsilon=V(x)\psi^\varepsilon+\lambda
    \varepsilon^\alpha\psi^\varepsilon \log
    |\psi^\varepsilon|^2\quad;\quad \psi^\eps_{\mid t=0}=\psi_0^\eps,
\end{equation}
where $\psi^{\varepsilon}=\psi^{\varepsilon}(t, x)$ is complex-valued,
$x \in \mathbb{R}^d$ with $d \ge 1$, $\lambda \in \mathbb{R}
\backslash\{0\}$, and $\alpha\ge 1$ (see the discussion
below). Throughout this paper, the
external potential $V=V(x)$ satisfies the 
following assumption:
\begin{hyp}\label{hyp:V}
  The external potential $V$ is smooth, real-valued, and at most quadratic:
  \begin{equation*}
 V\in C^\infty(\R^d;\R)\quad \text{and}\quad   \d^\beta V\in
 L^\infty\(\R^d\),\quad \forall |\beta|\ge 2.  
  \end{equation*}
\end{hyp}
To lighten notations, we consider $t\ge 0$ only, which is no
restriction as the equation is time reversible.

The logarithmic Schr\"odinger equation was introduced in \cite{BiMy76}, and has
been used since in various physical fields, such as 
quantum mechanics \cite{yasue},
quantum optics \cite{BiMy76,hansson}, nuclear
physics \cite{Hef85}, Bohmian mechanics \cite{DMFGL03}, effective
quantum gravity \cite{Zlo10}, theory of 
superfluidity and  Bose-Einstein condensation (BEC) \cite{BEC}.
As proposed in \cite{Zlo10,Zlo11}, the logarithmic nonlinearity may
extend quantum
mechanics thanks to a nonlinear model, likely to help understand
quantum gravity. In \cite{Bouharia2015}, the presence of an harmonic trap
was considered, in order to describe logarithmic BEC.
\smallbreak

We mention two mathematical properties associated with the logarithmic
nonlinearity in Schr\"odinger equations. First, and this is the main
reason why this model was introduced in \cite{BiMy76}, it is the only
nonlinearity that provides a tensorization property in the
multidimensional setting. Suppose $d\ge 2$ and that $V$ separates
coordinates, in the sense that
\begin{equation*}
  V(x) =\sum_{j=1}^d V_j(x_j).
\end{equation*}
If the initial datum $\psi_0^\eps$ is a tensor
product,
\begin{equation*}
  \psi_0^\eps(x) =\prod_{j=1}^d \varphi_j^\eps(x_j),
\end{equation*}
then the solution to \eqref{eq:logNLS} is given by
\begin{equation*}
  \psi^\eps(t,x) =\prod_{j=1}^d \psi^\eps_{j}(t,x_j), 
\end{equation*}
where each $\psi^\eps_j$ solves a one-dimensional equation,
\begin{equation*}
   i\eps\d_t \psi^\eps_j +\frac{\eps^2}{2} \d_{x_j}^2 \psi^\eps_j =V_j (x_j)\psi^\eps_j+ \lambda
   \eps^\alpha \psi_j^\eps \log|\psi_j^\eps|^2  \quad ;\quad
   \psi^\eps_{j\mid t=0} =\varphi^\eps_{j} . 
 \end{equation*}
The second property, unusual in a nonlinear setting as well, is that the
size of the solution is somehow irrelevant, in the sense that if, for
$k\in \C$, $\Psi_k^\eps$ denotes the solution to \eqref{eq:logNLS} with
initial datum $k\psi_0^\eps$ instead of $\psi_0^\eps$, then we have
\begin{equation*}
  \Psi_k^\eps(t,x) = k\psi^\eps(t,x) e^{-i\lambda \eps^{\alpha-1}t \log|k|^2}.
\end{equation*}
We will meet this property later, when considering the semiclassically
critical case for \eqref{eq:logNLS}.
\bigbreak

The semiclassical limit $\eps\to 0$ for \eqref{eq:logNLS} was
addressed by Ferriere \cite{FeBKW} for $\psi_0^\eps$ a WKB state (also
known as 
Lagrangian state), $\psi_0^\eps(x)= a_0(x)e^{i\phi_0(x)/\eps}$, in the
case $V=0$ with $\alpha=0$ (which corresponds to a supercritical case
as far 
as WKB analysis is concerned). In the case where $V$ is a bounded
potential, $\alpha=0$, $\lambda<0$, and $\psi_0^\eps$ is a
concentrating Gausson, 
\begin{equation*}
  \psi_0^\eps(x) = e^{ix\cdot p_0/\eps}R\(\frac{x-q_0}{\eps}\),\quad
  R(x) = e^{\frac{1+d}{2}}e^{\lambda |x|^2},
\end{equation*}
Ardila and Squassina \cite{ArSq18} proved that, locally uniformly in
time,  the solution remains
concentrated on the Gausson, and the center in phase
space evolves according to classical mechanics:
\begin{equation*}
  \psi^\eps(t,x) = e^{ix\cdot
    p(t)/\eps+i\theta^\eps(t)}R\(\frac{x-q(t)}{\eps}\)
  +w^\eps(t,x),\quad \eps^{-d/2}\|w^\eps(t)\|_{L^2(\R^d)}=\O(\eps), 
\end{equation*}
where $\theta^\eps(t)\in \R$ and 
\begin{equation*}
  \dot q(t) = p(t),\quad \dot p(t) = -\nabla V\(q(t)\)\quad ;\quad
  q(0)=q_0,\quad p(0)=p_0.
\end{equation*}
The above Hamilton system also plays a crucial role in the propagation
of coherent states considered in the present paper.
Indeed, we assume that the initial data $\psi^\eps_0$ is a
localized  wave packet   of the form
\begin{equation}\label{eq:CI}
\psi^\eps_0(x)=\eps^{-d/4} u_0\left(\frac{x-q_0}{\sqrt\eps}\right)
e^{i{(x-q_0)\cdot p_0/\eps}},\quad
q_0,p_0\in \R^d,
\end{equation}
and so the parameter $\alpha$ measures the strength of the nonlinear
interaction as $\eps\to 0$. Our goal is to describe the behavior of
$\psi^\eps$ as $\eps\to 0$, according to the value of $\alpha$. Note
that this initial datum is such that
$\|\psi_0^\eps\|_{L^2}=\|u_0\|_{L^2}$ is independent of $\eps>0$. In
view of the above remark, replacing the factor $\eps^{-d/4}$ with any
other power of $\eps$ leads to an explicit change in the expression of
$\psi^\eps(t,x)$ by a gauge transform. 
\smallbreak

In the linear case $\lambda=0$, the mathematical study of the
semiclassical limit $\eps\to 0$ in \eqref{eq:logNLS}--\eqref{eq:CI}
goes back to \cite{Hepp}, then developed in \cite{Hag80} and many
other contributions; we refer to \cite{robert2021coherent} for a
comprehensive presentation of the propagation of coherent states in
the linear case. The propagation of coherent for nonlinear
Schr\"odinger equations begins in \cite{CaFe11} (power nonlinearity)
and \cite{APPP11,CaCa11} (Hartree type nonlinearity). In this
nonlinear setting, the size of the initial data or, equivalently, the
size of the parameter in front of the nonlinearity, in terms of
$\eps$, is a crucial parameter in the leading order behavior of
$\psi^\eps$ as $\eps\to 0$. Below, we make this discussion precise in
the case of the logarithmic nonlinearity, and derive the notion of
criticality for  the parameter $\alpha$. 

\subsection{Notion of criticality}  
First, we proceed formally, and  we seek an approximate solution of
the equation\eqref{eq:logNLS} 

\begin{equation*}
\psi_{\mathrm{app}}^{\varepsilon}(t, x)=\frac{1}{\varepsilon^{d / 4}} u\left(t, \frac{x-q(t)}{\sqrt{\varepsilon}}\right) e^{i\phi(t,x)/\varepsilon},
\end{equation*}
where
\begin{equation}\label{eq:phi_lin}
  \phi(t,x)=\phi_{\rm lin}(t,x):=S(t)+p(t)\cdot(x-q(t)),
\end{equation}
and the rescaled space variable is 
$$
y=\frac{x-q(t)}{\sqrt{\varepsilon}}.
$$
Classically, we consider the Taylor expansion of $V(x)$ in terms of
the new space variable $y$ (see e.g. \cite{Carles2021book}):
\begin{equation*}
 V(x)= V\(q(t)+y\sqrt\eps\)=V(q(t))+\sqrt{\varepsilon} y \cdot \nabla V(q(t))+\frac{\varepsilon}{2}\left\langle y, \nabla^2 V(q(t)) y\right\rangle+\mathcal{O}\left(\varepsilon^{3 / 2}\right).
\end{equation*}
Plugging this expansion into \eqref{eq:logNLS}, we can measure how
close $\psi_{\rm app}^\eps$ is to solve \eqref{eq:logNLS}, by computing
\begin{equation}\label{eq:rewrite initial}
\begin{aligned}
    &i \varepsilon \partial_t \psi_{\mathrm{app}}^{\varepsilon}+\frac{\varepsilon^2}{2} \Delta \psi_{\mathrm{app}}^{\varepsilon}-V(x) \psi_{\mathrm{app}}^{\varepsilon}-\lambda\varepsilon^\alpha\psi_{\mathrm{app}}^{\varepsilon}\log |\psi_{\mathrm{app}}^{\varepsilon}|^2\\&=\frac{1}{\varepsilon^{d / 4}}\left(B_{0}+\sqrt{\varepsilon} B_{1}+\varepsilon B_{2} 
    +B_{\rm nl}+\O\(\varepsilon^{3/2}\)\right) e^{i\phi / \varepsilon},
\end{aligned}
\end{equation}
where $\phi$ is defined in \eqref{eq:phi_lin}, and
\begin{equation*}
  \begin{aligned}
&B_0  =-u\left(\dot{S}(t)-p(t) \cdot \dot{q}(t)+\frac{|p(t)|^2}{2}+V(q(t))\right), \\
&B_1 =-i(\dot{q}(t)-p(t)) \cdot \nabla u-y \cdot(\dot{p}(t)+\nabla V(q(t))) u, \\
&B_2 =i \partial_t u+\frac{1}{2} \Delta u-\frac{1}{2}\left\langle y, \nabla^2 V(q(t)) y\right\rangle u,\\
&B_{\rm nl}= -\lambda\varepsilon^\alpha u\log|u|^2+\lambda\frac{d}{2}u\varepsilon^\alpha\log\varepsilon.
\end{aligned}
\end{equation*}
As indicated by the notations, the nonlinear contribution is present
only in the term $B_{\rm nl}$.

When $\alpha>1$, $B_{\rm nl}$ is incorporated into the reminder term,
and we write
\begin{align*}
    &i \varepsilon \partial_t \psi_{\mathrm{app}}^{\varepsilon}+\frac{\varepsilon^2}{2} \Delta \psi_{\mathrm{app}}^{\varepsilon}-V(x) \psi_{\mathrm{app}}^{\varepsilon}-\lambda\varepsilon^\alpha\psi_{\mathrm{app}}^{\varepsilon}\log |\psi_{\mathrm{app}}^{\varepsilon}|^2\\&=\frac{1}{\varepsilon^{d / 4}}\left(B_{0}+\sqrt{\varepsilon} B_{1}+\varepsilon B_{2} +\O\(\varepsilon^{\min(3/2,\alpha)}\)\right) e^{i\phi / \varepsilon}.
\end{align*}
Due to the presence of the parameter $\eps$ in front of the factor
$\d_t \psi^\eps_{\rm app}$, the contribution of the remainder is
expected to be $\O\(\eps^{\min(1/2,\alpha-1)}\)$ in $L^2$, and we
  expect $\psi^\eps$ to be well approximated by the linear solution 
\begin{equation}\label{eq:evolution-linear}
i \varepsilon \partial_t \varphi^\varepsilon+\frac{\varepsilon^2}{2}
\Delta \varphi^\varepsilon=V(x) \varphi^\varepsilon\quad ;\quad
\varphi^\eps_{\mid t=0}=\psi_0^\eps. 
\end{equation}
We cancel out, successively, $b_1$, $b_0$ and $b_2$.
Setting $b_1=0$, we impose
\begin{equation*}
    \dot{q}(t)=p(t), \quad \dot{p}(t)=-\nabla V(q(t)) .
\end{equation*}
To agree with the initial coherent state \eqref{eq:CI}, we impose
$q(0)=q_0$ and $p(0)=p_0$, so that $(q, p)$ is given by the classical
Hamiltonian flow, 
\begin{equation}\label{eq:hamilton}
    \dot{q}=p, \quad \dot{p}=-\nabla V(q)\quad ; \quad q(0)=q_0, \quad p(0)=p_0,
\end{equation}
Setting $B_0=0$ and $S(0)=0$, we obtain the classical action (the action of classical mechanics).
\begin{equation}\label{eq:action}
    S(t)=\int_0^t\left(\frac{|p(s)|^2}{2}-V(q(s))\right) d s,
\end{equation}
Finally, setting $B_2=0$ and changing the notation $u$ to $v$ yields
\begin{equation}\label{eq:v}
    i \partial_t v+\frac{1}{2} \Delta v=\frac{1}{2}\left\langle y,
      \nabla^2 V(q(t)) y\right\rangle v \quad ; \quad v_{\mid t=0}=u_0. 
\end{equation}
To make notations consistent, and for future reference, we thus introduce
\begin{equation}\label{eq:phi_app}
    \varphi^\varepsilon_{\rm app}(t,x)=\frac{1}{\varepsilon^{d /
        4}}v\(t,\frac{x-q(t)}{\sqrt{\varepsilon}}\) e^{i\phi_{\rm lin}(t, x)/\eps},
\end{equation}
where $v$ solves \eqref{eq:v}. 

In the case $\alpha=1$, the contribution of $B_{\rm nl}$ is comparable
with that of $B_2$. The second term in $B_{\rm nl}$ may be viewed as a
constant (in $(t,x)$) potential, and is removed by a gauge
transform. This is a manifestation of the effect of the size in
logarithmic nonlinearity described above. We therefore modify the phase
$\phi$ and now write
\begin{equation}\label{eq:psiapp}
    {\psi}_{\rm app}(t,x)=\frac{1}{\varepsilon^{d /
        4}}u\(t,\frac{x-q(t)}{\sqrt{\varepsilon}}\)
    e^{i\phi_{\rm nl}(t, x)/\varepsilon},
\end{equation}
where 
\begin{equation*}
  \phi_{\rm nl}(t, x):=\phi_{\rm lin}(t,x) 
-\lambda\frac{d}{2}t\varepsilon\log\varepsilon= 
S(t)+p(t) \cdot(x-q(t))-\lambda\frac{d}{2}t\varepsilon\log\varepsilon.
\end{equation*}
Since we have not modified the expression of $B_0$ and $B_1$, $q$, $p$
and $S$ are still given by \eqref{eq:hamilton} and
\eqref{eq:action}. On the other hand, the envelope equation becomes
\begin{equation}\label{eq:u}
    i \partial_t u+\frac{1}{2} \Delta u=\frac{1}{2}\left\langle y,
      \nabla^2 V(q(t)) y\right\rangle u+\lambda u\log|u|^2 \quad ;
    \quad u_{\mid t=0}=u_0.
\end{equation}
 The above potential is quadratic in $y$. This implies, as we will see
 in the sequel, that, like in the linear
 case \eqref{eq:v} (as noticed initially in 
 \cite{Hel75}, see also \cite{Hag80,robert2021coherent}), the
 flow of this equation preserves Gaussian structures: if $a$ is a
 Gaussian, then so is $u(t,\cdot)$ for all $t\in \R$. In a nonlinear
 context, this property is a feature of the logarithmic
 nonlinearity. We note that this is also a feature of the coherent
 states approximation, that even if $V$ is not quadratic (in space),
 the potential in \eqref{eq:u} is. The potential in \eqref{eq:u}
 depends on time precisely when $V$ is not exactly quadratic, and the
 explicit computation for Gaussian data can be viewed as a feature of
 the coherent states approximation.
 \smallbreak
 
The case $\alpha>1$ will therefore be referred to as
\emph{subcritical}, and the case $\alpha=1$ as \emph{critical}. We
note that the case $\alpha<1$ seems to be out of reach in this setting
(as well as the supercritical case $\alpha<\alpha_c$ in \cite{CaFe11}
for power-like nonlinearity;
the case of Hartree type nonlinearity is different, as pointed out in
\cite{CaCa11} and \cite{Carles2021book}, and some supercritical cases can
be described in the case of a smooth, bounded kernel). 
\smallbreak

We emphasize that for a power-like nonlinearity $\eps^\alpha
|\psi^\eps|^{2\si}\psi^\eps$, $\si>0$, with
initial data like in \eqref{eq:CI}, the critical case is given by
$\alpha_c=1+\frac{d\si}{2}$, 
\begin{equation*}
    i \varepsilon\partial_t \psi^\varepsilon+\frac{\varepsilon^2}{2}
    \Delta \psi^\varepsilon=V(x)\psi^\varepsilon+
    \varepsilon^{1+d\si/2}|\psi^\varepsilon|^{2\si}\psi^\varepsilon 
    \quad;\quad \psi^\eps_{\mid t=0}=\psi_0^\eps,
\end{equation*}
so the critical case $\alpha=1$ in \eqref{eq:logNLS} may be viewed as
a natural candidate for the limit $\si\to 0$ (see also
\cite{WangZhang2019,GMS-p} for other
evidences in this direction). In the case of Hartree
equation with a homogeneous kernel ($\gamma>0$), the critical case is
(\cite{CaCa11}) 
\begin{equation*}
    i \varepsilon\partial_t \psi^\varepsilon+\frac{\varepsilon^2}{2}
    \Delta \psi^\varepsilon=V(x)\psi^\varepsilon+\lambda
    \varepsilon^{1+\gamma/2}\(\frac{1}{|x|^\gamma}\ast|\psi^\varepsilon|^2\)
    \psi^\varepsilon 
    \quad;\quad \psi^\eps_{\mid t=0}=\psi_0^\eps.
\end{equation*}
On the other hand, for Hartree type equations with a smooth kernel $K$
(bounded as well as its derivatives), the critical case is
(\cite{CaCa11}), like for \eqref{eq:logNLS},
\begin{equation*}
    i \varepsilon\partial_t \psi^\varepsilon+\frac{\varepsilon^2}{2}
    \Delta \psi^\varepsilon=V(x)\psi^\varepsilon+\lambda
    \varepsilon\(K\ast|\psi^\varepsilon|^2\)
    \psi^\varepsilon 
    \quad;\quad \psi^\eps_{\mid t=0}=\psi_0^\eps.
\end{equation*}

\subsection{Main results}
Denote
\begin{equation*}
  \Sigma^k=H^k \cap \mathcal{F}\left(H^k\right)=\left\{f \in
    H^k\left(\mathbb{R}^d\right), \quad x \mapsto|x|^k f(x) \in
    L^2\left(\mathbb{R}^d\right)\right\} ,
\end{equation*}
endowed with natural norm.
When $k=1$, we simply denote $\Sigma^1$ by $\Sigma$.

Our first result shows that indeed, the case $\alpha>1$ is
subcritical, in the sense that $\psi^\eps$ is well approximated by the
linear solution $\varphi^\eps$, up to some Ehrenfest
time (a time of order $\log\frac{1}{\eps}$). The solution
$\varphi^\eps$ is also well approximated by $\varphi_{\rm app}^\eps $
up to some Ehrenfest
time. 

\begin{proposition}[Subcritical case]\label{prop:subcritical}
  Suppose that $\alpha>1$. If $u_0\in \Sigma$, then
  for any $0<\delta<\alpha-1$, there exists $C>0$ independent of $\eps$ such that,
  with $\varphi^\eps$ solution to the linear equation
  \eqref{eq:evolution-linear}, we have 
  \begin{equation*}
    \|\psi^\eps(t)-\varphi^\eps(t)\|_{L^2(\R^d)}\lesssim
  \eps^{\alpha-1-\delta}e^{Ct},\quad t\ge 0.
\end{equation*}
In particular, there exists $c_0>0$ independent of
$\eps$  such that
\begin{equation*}
  \sup_{0\le t\le c_0\log\frac{1}{\eps}}\|\psi^\eps(t)-
  \varphi^\eps(t)\|_{L^2(\R^d)}\Tend \eps 0 0 .  
\end{equation*}
If in addition $u_0\in \Sigma^2$,  there exists $C>0$ independent of
$\eps$ such that, with $\varphi_{\rm app}^\eps$ defined in
\eqref{eq:phi_app},  $v$ given by \eqref{eq:v} and $\phi_{\rm lin}$ by
\eqref{eq:phi_lin}, we have
\begin{equation*}
  \|\varphi^\eps(t)-\varphi_{\rm app}^\eps(t)\|_{L^2(\R^d)}\lesssim
  \eps^{1/4}e^{Ct},\quad t\ge 0.
\end{equation*}
Finally, if $u_0\in \Sigma^3$,  there exists $C>0$ independent of
$\eps$ such that
\begin{equation*}
  \|\varphi^\eps(t)-\varphi_{\rm app}^\eps(t)\|_{L^2(\R^d)}\lesssim
  \sqrt \eps e^{Ct},\quad t\ge 0.
\end{equation*}
\end{proposition}
In the linear case $\lambda=0$, Ehrenfest time,
  of the form 
  $c\log\frac{1}{\eps}$, may actually correspond to the maximum time
  of validity of the quadratic approximation  $\varphi^\eps$ by
  $\varphi_{\rm app}^\eps$
  (for some best constant $c>0$),
  see e.g. \cite{robert2021coherent}. The last estimate of
  Proposition~\ref{prop:subcritical} is classical, 
  and is essentially sharp (up to the fact that the constant $C$ is
  not explicit here);  see
  e.g. \cite{Carles2021book,robert2021coherent}. Also, the order of
  magnitude $\eps^{\alpha-1-\delta}$ in the error $\|\psi^\eps(t)- 
  \varphi^\eps(t)\|_{L^2(\R^d)}$ is expected to be essentially (due to
  the restriction $\delta>0$) sharp. The
  intermediary error estimate for $\varphi^\eps-\varphi_{\rm app}^\eps$
  when  $u_0\in \Sigma^2$ appears to be
  new; its main interest will show in the critical case $\alpha=1$,
  where the argument of proof is adapted, 
  and it is not clear whether even at the level of
  Proposition~\ref{prop:subcritical} this estimate is sharp.
We now come to critical phenomena, and set $\alpha=1$. The case of
Gaussian initial profiles turns out to be more convenient, even though
we consider a nonlinear setting; see also Remark~\ref{rem:sharpness}.
\begin{theorem}[Critical case]\label{theo:single}
  Suppose that $\alpha=1$. \\
  \noindent $(1)$ If $u_0\in \Sigma^2$, then with $\psi_{\rm app}^\eps$
defined in \eqref{eq:psiapp} ($u$ is given by \eqref{eq:u}), there exists
 $C>0$ independent of $\eps$ such that,
  \begin{equation*}
    \|\psi^\eps(t)-\psi_{\rm app}^\eps(t)\|_{L^2(\R^d)}\lesssim
  \eps^{1/4}e^{e^{Ct}},\quad t\ge 0.
\end{equation*}
In particular, there exists $c$ independent of $\eps $ such that
\begin{equation*}
  \sup_{0\le t\le c\log\log\frac{1}{\eps}}\|\psi^\eps(t)-
  \psi_{\rm app}^\eps(t)\|_{L^2(\R^d)}\Tend \eps 0 0 .  
\end{equation*}
  \noindent $(2)$ Suppose that $V$ decouples variables,
  \begin{equation*}
    V(x) = \sum_{j=1}^d V_j(x_j).
  \end{equation*}
 If  $u_0$ is a Gaussian of the form
  \begin{equation*}
    u_0(y) = b_0 \exp \(-\frac{1}{2}\sum_{j=1}^da_{0j}y_j^2\), \quad
    a_{0j},b_0\in \C,\ \RE a_{0j}>0,
  \end{equation*}
 then  there exists
 $C>0$ independent of $\eps$ such that,
  \begin{equation*}
    \|\psi^\eps(t)-\psi_{\rm app}^\eps(t)\|_{L^2(\R^d)}\lesssim
  \sqrt \eps e^{Ct},\quad t\ge 0.
\end{equation*}
\end{theorem}
We remark that in the second case, Assumption~\ref{hyp:V} implies that
each $V_j$ satisfies (the one-dimensional version of)
Assumption~\ref{hyp:V}. If $V$ does not decouple variables, the
computations presented in Section~\ref{sec:gaussian} become more
involved, as one has to deal with time dependent matrices instead of
scalars, since even if $u_0$ decouples variables like above, this
property is lost for $t>0$ in general. We choose to present this simplified
setting to keep notations as light as possible. 

\begin{theorem}[Nonlinear superposition]\label{theo:superp}
Suppose that $\alpha=1$,  that $V$ decouples variables,
  \begin{equation*}
    V(x) = \sum_{j=1}^d V_j(x_j),
  \end{equation*}
and that the initial datum in \eqref{eq:logNLS}  is the sum of two
coherent states,
\begin{equation}\label{eq:CI2}
\psi^\eps_0(x)=\eps^{-d/4} \sum_{j=1}^2u_{0j}\left(\frac{x-q_{0j}}{\sqrt\eps}\right)
e^{i{(x-q_{0j})\cdot p_0/\eps}},\quad
q_{0j},p_{0j}\in \R^d,
\end{equation}
with different centers in phase space, $(q_{01},p_{01})\not =
(q_{02},p_{02})$. If $u_{01}$ and $u_{02}$ are two Gaussian functions
like in the second part of Theorem~\ref{theo:single}, then uniformly
on bounded time intervals, a nonlinear superposition holds: if
$\psi_{j,\rm app}^\eps$ denotes the approximate solution associated
with $u_{0j}$, then for any $T>0$ and any $\gamma<1/2$,
\begin{equation*}
 \sup_{t\in [0,T]} \left\| \psi^\eps(t) - \psi_{1,\rm app}^\eps(t)
-\psi_{2,\rm app}^\eps(t)\right\|_{L^2}\lesssim \eps^{\gamma}. 
\end{equation*}
If in addition $d=1$, and $E_1\not=E_2$, where 
\begin{equation*}
  E_j = \frac{p_{0j}^2}{2}+V\(q_{0j}\),
\end{equation*}
then for all $\gamma<1/2$, there exists $C>0$ independent of $\eps\in
(0,1)$ such that 
\begin{equation*}
\left\| \psi^\eps(t) - \psi_{1,\rm app}^\eps(t)
-\psi_{2,\rm app}^\eps(t)\right\|_{L^2}\lesssim
\eps^{\gamma}e^{Ct},\quad \forall t\ge 0. 
\end{equation*}
In particular, there exists $c>0$ independent of $\eps$ such that
\begin{equation*}
 \sup_{0\le t\le c\log\frac{1}{\eps}} \left\| \psi^\eps(t) - \psi_{1,\rm app}^\eps(t)
-\psi_{2,\rm app}^\eps(t)\right\|_{L^2}\Tend \eps 0 0 . 
\end{equation*}
\end{theorem}
The above result is readily generalized to the case of finitely many
initial Gaussian coherent states, with pairwise distinct centers in
phase space. We restrict the presentation to the case of two coherent
states, to make notations as light as possible.
\begin{remark}[On sharpness in Theorems~\ref{theo:single} and
  \ref{theo:superp}]\label{rem:sharpness}
  The first part of Theorem~\ref{theo:single} corresponds to the most
  general framework we could consider in the critical case $\alpha=1$,
  both in terms of initial data (simply $u_0\in \Sigma^2$) and in
  terms of the external potential $V$ (satisfying
  Assumption~\ref{hyp:V}). We do not expect this error estimate to be
  sharp: the double exponential growth in time is probably not sharp,
  and, as pointed out above for the linear case, it is not clear
  whether the $\frac{1}{4}$ 
  power for $\varepsilon$ is sharp. On the other hand, in the special
  case where $u_0$ is Gaussian, and $V$ decouples variables (which is
  automatic if $d=1$), then the second error estimate in
  Theorem~\ref{theo:single} is expected to be sharp. Likewise, up to
  the restriction $\gamma<1/2$ (the case $\gamma=1/2$ might be
  admissible, this question remains open), the conclusion of
  Theorem~\ref{theo:superp} is expected to be sharp. It is likely
  that a similar superposition principle holds for more general
  initial envelopes $u_{0j}$ (not only Gaussians), and potentials $V$ mixing variables, but
  we could not find a method to address such a framework.
 \end{remark}
\subsection{Content}
\label{sec:content}

In Section~\ref{sec:prelim}, we recall several general estimates,
address the Cauchy problems for \eqref{eq:logNLS} and \eqref{eq:u},
proving new estimates in the case of \eqref{eq:u}, in
Section~\ref{sec:further}. Proposition~\ref{prop:subcritical} is proven in
Section~\ref{sec:subcritical}, Theorem~\ref{theo:single} in
Section~\ref{sec:critical}, and Theorem~\ref{theo:superp} in
Section~\ref{sec:superp}. 
\subsection{Notations}
\label{sec:notations}

Throughout the paper,
$C$ denotes a constant independent of $t$ and $\eps\in (0,1]$,
which may change from 
one line to the other. 

For $f,g\ge0$, we write $f\lesssim g$ whenever there exists a
constant $C$ independent of $t\ge 0$, $x\in \R^d$, and $\eps\in (0,1)$,
such that $f\le Cg$.

\section{Preliminary results}
\label{sec:prelim}

\subsection{Some technical tools}
We recall that under Assumption~\ref{hyp:V}, Cauchy-Lipschitz Theorem
implies: 
\begin{lemma}\label{lem:hamilton}
  Let $(q_0,p_0)\in \R^d\times \R^d$. Under Assumption~\ref{hyp:V},
\eqref{eq:hamilton} has a unique global, smooth solution 
$(q,p)\in 
C^\infty(\R;\R^d)^2$. It grows at
most exponentially:    
 \begin{equation*}
 \exists C_0>0,\quad \left|q(t)\right|+\left|p(t)\right|\lesssim
 e^{C_0t},\quad \forall t\ge 0.
 \end{equation*}
\end{lemma}
The following lemma, originating in \cite{Cazenave1980}, is crucial in
the study of Schr\"odinger equations with a logarithmic nonlinearity,
and will be used several times in the sequel:
\begin{lemma}[From Lemma~1.1.1 in \cite{Cazenave1980}]\label{lem:CH}
  There holds
  \begin{equation*}
    \left|\operatorname{Im}\left(\left(z_2 \log
 \left|z_2\right|^2-z_1 \log
 \left|z_1\right|^2\right)\left(\overline{z_2}-
 \overline{z_1}\right)\right)\right| 
    \le 2\left|z_2-z_1\right|^2, \quad \forall z_1, z_2 \in
    \mathbb{C}.
  \end{equation*}
\end{lemma}
We recall the Gagliardo-Nirenberg inequality and a standard inequality
(see e.g. (2.3) and its proof in \cite{CaGa18}),
\begin{align}\label{ineq:GN}
    &\|v\|_{L^p}\lesssim \|v\|_{L^2}^{1-\theta(p)}\|\nabla
    v\|_{L^2}^{\theta(p)},\quad \text{for}\quad 2\le p<
      \frac{2d}{(d-2)_+},\\
& \left\|v\right\|_{L^{p}} \lesssim\left\| v\right\|_{L^2}^{1-\theta(p')}
  \left\| y v\right\|_{L^2}^{\theta(p')},\quad \text{for}\quad
  \max\(1,\frac{2d}{d+2}\)< p\le 2, \label{eq:GNdual}
\end{align}
where
\begin{equation*}
  \theta(p)=d\(\frac{1}{2}-\frac{1}{p}\),\quad
  \frac{1}{p}+\frac{1}{p'}=1. 
\end{equation*}
\subsection{Cauchy problem}
\label{sec:cauchy}

In the case without potential, $V=0$, the Cauchy problem for
\eqref{eq:logNLS} was addressed initially in \cite{Cazenave1980} in
the case $\lambda<0$, then resumed in \cite{GuLoNi10} in the case
$\lambda\in \R$, with a generalization in \cite{CaGa18}, using
compactness methods. The space in which the Cauchy problem is solved
was further enlarged in \cite{CHO24,HO25}, using more constructive
methods. In the case where $\lambda<0$ and $V$ is a harmonic potential, the Cauchy
problem for \eqref{eq:logNLS} was solved in $\Sigma$ in \cite{ACS20},
by adapting the approach from \cite{Cazenave1980}. In the framework
considered in the present paper ($V$ satisfying Assumption~\ref{hyp:V}
and $\lambda\in \R$), the Cauchy problem was solved in $\Sigma$ in
\cite{CaFe21}:
\begin{proposition}[Proposition~1.3 in \cite{CaFe21}]\label{prop:cauchy}
 Let $d\ge 1$, $\lambda,\alpha\in \R$, $\eps>0$, and $V$ satisfying
 Assumption~\ref{hyp:V}. For
  $\psi^\eps_0\in \Sigma$, \eqref{eq:logNLS} has a unique solution
  $\psi^\eps\in L^\infty_{\rm  loc}(\R;\Sigma)\cap C(\R;L^2(\R^d))$.
Moreover, the following quantities are independent of time:
\begin{align*}
  &\text{Mass: } M^\eps(t)=\|\psi^\eps(t)\|_{L^2(\R^d)}^2\equiv M^\eps(0),\\
    &\text{Energy: } E^\eps(t)=\frac{1}{2}\|\eps\nabla
     \psi^\eps(t)\|_{L^2(\R^d)}^2
     +\int_{\R^d}V(x)|\psi^\eps(t,x)|^2dx\\
   &\phantom{\text{Energy}: E^\eps(t)=}
     +\lambda \eps^\alpha \int_{\R^d} |\psi^\eps(t,x)|^2\log
     |\psi^\eps(t,x)|^2dx\equiv E^\eps(0). 
  \end{align*}
\end{proposition}
We recall the main steps of the proof, to show that the above result
is readily extended to the case of a time dependent potential like in
\eqref{eq:u}. Since the role of $\eps>0$ is absent at this level of
discussion, we set $\eps=1$ in order to lighten notations. 
In view of the singularity of the logarithm at the
origin, we consider the family of regularized equations, for
$\delta>0$,
\begin{equation}\label{eq:logNLS-app}
  i\d_t \psi^\delta+\frac{1}{2}\Delta \psi^\delta = V(x)\psi^\delta+\lambda
  \psi^\delta\log\(\delta+|\psi^\delta|^2\)\quad;\quad
  \psi^\delta_{\mid t=0}=\psi_0.  
\end{equation}
For fixed $\delta>0$, the nonlinearity is smooth, with moderate growth
as $|\psi^\delta|\to \infty$ (the nonlinearity is $L^2$-subcritical in any
dimension), so under Assumption~\ref{hyp:V}, global existence in
$\Sigma$ follows for instance from \cite{Ca11} (where a dependence of
$V$ upon time is allowed, like in \eqref{eq:u}): $\psi^\delta, \nabla
\psi^\delta, x\psi^\delta \in C(\R;L^2(\R^d))$.

The sequence $(\psi^\delta)_{0<\delta\le 1}$ converges, thanks to compactness
arguments based on uniform a priori estimates. First, the $L^2$-norm
of $\psi^\delta $ is independent of
time, $\|\psi^\delta(t)\|_{L^2}=\|\psi_0\|_{L^2}$. For $1\le j\le d$,
differentiating \eqref{eq:logNLS-app} with respect to $x_j$ yields
\begin{align*}
  i{\partial}_t {\partial}_j \psi^\delta +\frac{1}{2} \Delta
  {\partial}_j\psi^\delta &=V(x)\d_j \psi^\delta + \d_j V(x)
 \psi^\delta+\lambda \log\left(\delta+|\psi^\delta|^2\right){\partial}_j\psi^\delta
  \\
  &\quad + 2\lambda \frac{1}{\delta+|\psi^\delta|^2}  \RE (
    \overline{\psi^\delta} {\partial}_j\psi^\delta) \psi^\delta.
\end{align*}
By Assumption~\ref{hyp:V}, $|\d_j V(x)|\lesssim 1+|x|$, so the standard
 $L^2$ estimate yields
\begin{equation*}
  \frac{d}{dt}\|\nabla \psi^\delta(t)\|_{L^2}^2\le C \(
  \|\psi^\delta(t)\|_{L^2}^2+\|x \psi^\delta(t)\|_{L^2}^2 + \|\nabla
  \psi^\delta(t)\|_{L^2}^2\),
\end{equation*}
where $C$ is independent of $\delta$. Similarly, 
\begin{align*}
  i{\partial}_t\( x_j \psi^\delta \)+\frac{1}{2} \Delta
 (x_j\psi^\delta) &= \d_j \psi^\delta+V(x)x_j \psi^\delta  +\lambda
               \ln\left(\delta+|\psi^\delta|^2\right)x_j\psi^\delta, 
\end{align*}
hence
\begin{equation*}
  \frac{d}{dt}\|x \psi^\delta(t)\|_{L^2}^2\le 2 \int_{\R^d} |x \psi^\delta(t,x)|
  |\nabla \psi^\delta(t,x)|dx\le\|x \psi^\delta(t)\|_{L^2}^2 + \|\nabla
  \psi^\delta(t)\|_{L^2}^2. 
\end{equation*}
In view of the conservation of the mass, Gr\"onwall lemma implies that
there exists $C$ independent of $\delta$ 
such that
\begin{equation*}
  \|x \psi^\delta(t)\|_{L^2}^2 + \|\nabla
  \psi^\delta(t)\|_{L^2}^2 \le C\(\|u_0\|_{L^2}^2+\|x u_0\|_{L^2}^2 + \|\nabla
  u_0\|_{L^2}^2 \)  e^{Ct}, \quad t\ge 0. 
\end{equation*}
Therefore, we have compactness in space for the sequence
$(\psi^\delta)_{0<\delta\le 1}$. Compactness in time follows from
\eqref{eq:logNLS-app}. Arzela-Ascoli theorem yields  a converging
subsequence, hence the existence part of 
Proposition~\ref{prop:cauchy}.
\smallbreak

Uniqueness follows from the remark that any solution $\psi\in
L^\infty_{\rm loc}(\R;\Sigma)$ to \eqref{eq:logNLS} actually
belongs to $C(\R;L^2(\R^d))$. The standard $L^2$-estimate for the
difference of two solutions $\psi$ and $\phi$ (with a common initial datum
$\psi_0$) yields
\begin{equation*}
  \frac{d}{dt}\|\psi-\phi\|_{L^2}^2 \le 2\left|\IM \int \((\bar
    \psi-\bar \phi) \(\psi\log|\psi|^2 - \phi\log|\phi|^2\)\)\right|.
\end{equation*}
Lemma~\ref{lem:CH} and Gr\"onwall lemma then imply
$\psi\equiv\phi$. The conservations of mass and energy are justified
by the same arguments as in \cite{CazCourant}, for instance. 

Repeating the above arguments, we have a similar result in the case of
\eqref{eq:u}. 
The
fact that the potential may depend on time implies that the energy is
no longer conserved, but, as we do not use the notion of energy here,
we leave out the adapted statement regarding this aspect:
\begin{lemma}\label{lem:u}
  Let $d\ge 1$, $V$ satisfying Assumption~\ref{hyp:V}, and $\lambda\in
  \R$. For any $u_0\in \Sigma$, \eqref{eq:u} has a unique solution $u\in
  L^\infty_{\rm loc}(\R;\Sigma)\cap C(\R;L^2(\R^d))$. In addition,
  $\|u(t)\|_{L^2}=\|u_0\|_{L^2}$ for all $t\in \R$, and
  there exists $C>0$ such that
  \begin{equation*}
    \|\nabla u(t)\|_{L^2}+ \|y u(t)\|_{L^2}\lesssim e^{Ct},\quad
    \forall t\ge 0. 
  \end{equation*}
\end{lemma}
We note that in this statement, we say nothing about $q(t)$: this is
due to the fact that Assumption~\ref{hyp:V} implies $\nabla^2V\in
L^\infty(\R^d))$, which is the only relevant piece of information in
the argument of the proof. 
\subsection{Further estimates in the critical case}
\label{sec:further}

To prove Theorem~\ref{theo:single}, we will need higher
localization property than in Lemma~\ref{lem:u}, namely
$|y|^2u\in L^\infty_{\rm loc}(\R;L^2(\R^d))$. Due to the presence of
the quadratic potential in \eqref{eq:u}, the regularity is expected to
be the same in space and frequency, so it is natural to work in
$\Sigma^2$.

Unlike in the case without potential (\cite{CaGa18}), it is not clear
that the flow associated to \eqref{eq:u} preserves the $H^2$
regularity. Indeed, the argument of \cite{CaGa18} consists in adapting
the proof from \cite{CazCourant} in the case of power-like
nonlinearities, where one first proves that $\d_t u \in L^\infty_{\rm
    loc}(\R;L^2(\R^d))$, and then uses the equation to infer that
  $\Delta u$ enjoys the same regularity. Here, we face two new
  difficulties. First, even if we know that $\d_t u \in L^\infty_{\rm
    loc}(\R;L^2(\R^d))$, the possible lack of ellipticity of the
  operator
  \begin{equation*}
    -\frac{1}{2}\Delta + \frac{1}{2}\<y,\nabla^2 V\(q(t)\)y\>
  \end{equation*}
  makes it difficult, if not impossible, to infer that $\Delta
  u,|y|^2u \in L^\infty_{\rm
    loc}(\R;L^2(\R^d))$. Typically, one can think of $V(x)=
  -\omega^2|x|^2/2$, a case where $\nabla^2 V\(q(t)\)=-\omega^2I_d$ is
  constant. In the case where $\nabla^2 V\(q(t)\)$ is not constant, we
  also have to deal with the time derivative of this term, which
  appears in the equation satisfied by $\d_t u$, leading to a term
  behaving like $\dot q(t)|y|^2u$, and so estimating $\d_t u$ and
  $|y|^2 u$ are two connected questions. It is actually the presence
  of this term, and the possible exponential growth of $\dot q$, which
  explains the double exponential in the following statement. 
\begin{proposition}\label{prop:Sigma2}
  Let $u_0\in \Sigma^2$. 
The solution to \eqref{eq:u}, $u \in
L^\infty_{\rm loc}(\R;\Sigma)$,  enjoys the extra regularity $u \in
L^\infty_{\rm loc}(\R;\Sigma^2)$. 
In addition, there exists $C$ independent of $t\ge 0$ such that
\begin{equation*}
   \int_{\R^d}\<y\>^4|u(t,y)|^2dy + \int_{\R^d}|\Delta 
  u(t,y)|^2dy \le C e^{e^{Ct}},\quad \forall t\ge 0. 
\end{equation*}
\end{proposition}
\begin{proof}
  We recall that $u$ is constructed as the limit of the family
  $(u^\delta)_{0<\delta\le 1}$ as $\delta \to 0$, where $u^\delta$
  solves
 \begin{equation}
  \label{eq:udelta}
  i\d_t u^\delta +\frac{1}{2}\Delta u^\delta = \frac{1}{2}\<y,\nabla^2
  V\(q(t)\)y\>u^\delta 
  + \lambda u^\delta\log\(\delta+|u^\delta|^2\),\quad u^\delta_{\mid t=0}=u_0.
\end{equation}
We prove that there exists $C$ independent of $t\ge 0$ and $\delta\in
(0,1)$ such that
\begin{equation}\label{eq:doubleexpdelta}
  \int_{\R^d}\<y\>^4|u^\delta(t,y)|^2dy + \int_{\R^d}|\Delta
  u^\delta(t,y)|^2dy \le C e^{e^{Ct}}. 
\end{equation}
The equation satisfied by $|y|^2 u^\delta $ is
\begin{align*}
  i\d_t (|y|^2u^\delta) +\frac{1}{2}\Delta (|y|^2u^\delta)
  &=
  \frac{1}{2}\left[\Delta,|y|^2\right]u^\delta+
    \frac{1}{2}\<y,\nabla^2
  V\(q(t)\)y\>|y|^2u^\delta \\
  &\quad + \lambda |y|^2u^\delta\log\(\delta+|u^\delta|^2\).
\end{align*}
By direct computations,
\begin{equation*}
  \left[\Delta,|y|^2\right] = 4y\cdot \nabla +2d, 
\end{equation*}
so the standard energy estimate for Schr\"odinger equations yields
\begin{equation}\label{eq:borne1}
  \frac{d}{dt}\||y|^2 u^\delta\|_{L^2} \lesssim \|y\cdot\nabla
  u^\delta\|_{L^2} + \|u^\delta\|_{L^2}.
\end{equation}
We therefore  examine the equation satisfied by $y\cdot\nabla
u^\delta$. First, $\nabla u^\delta$ solves
\begin{align*}
  i\d_t \nabla u^\delta +\frac{1}{2}\Delta \nabla u^\delta 
  &=\(\nabla^2  V\(q(t)\)y\)u^\delta+
    \frac{1}{2}\<y,\nabla^2
  V\(q(t)\)y\>\nabla u^\delta \\
  &\quad + \lambda \(\nabla u^\delta\)\log\(\delta+|u^\delta|^2\)+2\lambda
    \frac{u^\delta\RE\(\overline u^\delta \nabla
    u^\delta\)}{\delta+|u^\delta|^2}. 
\end{align*}
Taking the inner product with $y$ yields
\begin{align*}
  i\d_t y\cdot \nabla u^\delta +\frac{1}{2}\Delta \(y\cdot \nabla u^\delta \)
  &=\frac{1}{2}\left[\Delta,y\right]\cdot \nabla u^\delta +  \<y,\nabla^2
    V\(q(t)\)y\>u^\delta\\
  &\quad+ 
    \frac{1}{2}\<y,\nabla^2
  V\(q(t)\)y\>y\cdot \nabla u^\delta \\
  &\quad + \lambda \(y\cdot \nabla u^\delta\)\log\(\delta+|u^\delta|^2\)+2\lambda
    \frac{u^\delta\RE\(\overline u^\delta y\cdot \nabla
    u^\delta\)}{\delta+|u^\delta|^2},
\end{align*}
and we note that
\begin{equation*}
  \frac{1}{2}\left[\Delta,y\right]\cdot \nabla u^\delta= \Delta u^\delta,
\end{equation*}
so by energy estimate, recalling that $\nabla^2 V\in L^\infty(\R^d)$,
\begin{equation}\label{eq:borne2}
   \frac{d}{dt}\|y\cdot\nabla  u^\delta\|_{L^2} \lesssim \|\Delta
  u^\delta\|_{L^2} + \||y|^2 u^\delta\|_{L^2} +\|y\cdot\nabla u^\delta\|_{L^2}.
\end{equation}
To bound the term $\Delta u^\delta$, we note that differentiating the
(regularized) logarithmic nonlinearity makes it impossible to get
estimates which are uniform in $\delta\in (0,1)$, so we use directly
the information given by \eqref{eq:udelta},
\begin{equation}\label{eq:estDelta}
  \|\Delta  u^\delta\|_{L^2}\lesssim \|\d_t u^\delta\|_{L^2} +
  \||y|^2u^\delta\|_{L^2} +
  \left\|u^\delta\log\(\delta+|u^\delta|^2\)\right\|_{L^2}. 
\end{equation}
Note the inequality
\begin{align*}
  \left|u^\delta\log\(\delta+|u^\delta|^2\)\right|^2
  &\lesssim |u^\delta|^2 \( \(\delta+|u^\delta|^2\)^{\eta/2}+
    \(\delta+|u^\delta|^2\)^{-\eta/2}\)\\
 &\lesssim |u^\delta|^{2+\eta} + |u^\delta|^{2-\eta}, 
\end{align*}
for any $\eta>0$, where the implicit multiplicative constant depends
on $\eta>0$ but not on $\delta\in (0,1)$. For $\eta>0$ sufficiently small,
this implies, in view of (the proof of) Lemma~\ref{lem:u}, since $\Sigma
\hookrightarrow L^{2+\eta}\cap L^{2-\eta}$ provided that $\eta>0$ is
sufficiently small (see \eqref{ineq:GN} and \eqref{eq:GNdual}),
\begin{equation}\label{eq:estlog}
  \left\|u^\delta\log\(\delta+|u^\delta|^2\)\right\|_{L^2}\lesssim
  e^{Ct},\quad t\ge 0.
\end{equation}
To close the system of estimates \eqref{eq:borne1}-\eqref{eq:borne2},
we estimate $\d_t u^\delta$ in $L^2$. For
$u_0\in \Sigma^2$, $\d_t u^\delta$ solves an equation rather similar
to the one satisfied by
$\nabla u^\delta$, but we must pay attention to the initial data,
given by \eqref{eq:udelta} at $t=0$:
\begin{align*}
  i\d_t \(\d_t u^\delta\) +\frac{1}{2}\Delta \(\d_t u^\delta\)
  &=
  \frac{1}{2}\<y,\nabla^2 
    V\(q(t)\)y\>\d_t u^\delta 
  +\frac{1}{2}\(\nabla^3 
  V\(q(t)\)\cdot y\cdot y\cdot \dot q(t)\) u^\delta\\
  &\quad   + \lambda \d_t u^\delta\log\(\delta+|u^\delta|^2\)
    + 2\lambda
   \frac{u^\delta \RE \(\overline u^\delta\d_t
    u^\delta\)}{|\delta+|u^\delta|^2},\\
  \d_t u^\delta_{\mid t=0}
  & = \frac{i}{2}\Delta u_0 -\frac{i}{2}\<y,\nabla^2 V(q_0)y\>u_0
   -i\lambda u_0\log\(\delta+|u_0|^2\).
\end{align*}
Recalling that the third derivatives of $V$ are bounded and that $\dot
q(t)=p(t)$ grows at most exponentially in time
(Lemma~\ref{lem:hamilton}),
energy estimate yields
\begin{equation}\label{eq:borne3}
   \frac{d}{dt}\|\d_t u^\delta\|_{L^2} \lesssim
   e^{Ct}\||y|^2u^\delta\|_{L^2} + \|\d_t u^\delta\|_{L^2} .
 \end{equation}
 We also note that in view of \eqref{eq:u}, $\| \d_t u^\delta_{\mid t=0}\|_{L^2}\le C\(
 \|u_0\|_{\Sigma^2}\)$. Summing \eqref{eq:borne1}, \eqref{eq:borne2}
 and \eqref{eq:borne3}, we find, using \eqref{eq:estDelta} and \eqref{eq:estlog},
 \begin{align*}
   \frac{d}{dt}\( \||y|^2u^\delta\|_{L^2} +  \|y\cdot\nabla
   u^\delta\|_{L^2} +\|\d_t u^\delta\|_{L^2} \)
   &\lesssim
     e^{Ct}\||y|^2u^\delta\|_{L^2} + \|\d_t u^\delta\|_{L^2} \\
   &\quad + \|y\cdot\nabla
     u^\delta\|_{L^2} +e^{Ct}.
 \end{align*}
 Gr\"onwall lemma yields
 \begin{equation*}
   \||y|^2u^\delta\|_{L^2} +  \|y\cdot\nabla
   u^\delta\|_{L^2} +\|\d_t u^\delta\|_{L^2} \le
   C\(\|u_0\|_{\Sigma^2}\) e^{e^{Ct}},
 \end{equation*}
 and invoking \eqref{eq:estDelta} and \eqref{eq:estlog},
 \eqref{eq:doubleexpdelta} follows. Since we already know that
 $u^\delta$ converges to $u$, the proposition stems from Fatou's
 lemma.  
\end{proof}

\section{Subcritical case}
\label{sec:subcritical}

In this section, we prove Proposition~\ref{prop:subcritical}. We start
by recalling some properties of the solution $\varphi^\eps$ to
\eqref{eq:evolution-linear}.

\subsection{Properties of $\varphi_{\rm app}^\varepsilon$ and $\varphi^\eps$}
We resume the approximate solution of the linear case,
\begin{equation*}
    \varphi^\varepsilon_{\rm app}(t,x)=\frac{1}{\varepsilon^{d /
        4}}v\(t,\frac{x-q(t)}{\sqrt{\varepsilon}}\) e^{i\phi_{\rm
        lin}(t, x)/\eps },
\end{equation*}
where $\phi_{\rm lin}$ is defined in \eqref{eq:phi_lin}, and infer
that for any $1<p<\infty$, 
\begin{equation*}
    \|\varphi^\varepsilon_{\rm app}(t)\|_{L^p_x}=\varepsilon^{\frac{d}{2}\(\frac{1}{p}-\frac{1}{2}\)}\|v(t)\|_{L^p_y}.
  \end{equation*}
In the linear case, it is easy to prove (see e.g. \cite{Carles2021book}):
\begin{lemma}\label{lem:(exp)_k}
  Let $k\ge
  1$. If $u_0\in
  \Sigma^k$, there exists $C>0$ such that the solution $v\in L^\infty_{\rm
    loc}(\R;\Sigma^k)$ to \eqref{eq:v} satisfies
  $\|v(t)\|_{L^2}=\|u_0\|_{L^2}$ and, for all $\beta
  \in \N^d$ with $1\le |\beta|\le k$, 
  \begin{equation*}
    \|\d_y^\beta v(t)\|_{L^2(\R^d)}+ \|y^\beta
    v(t)\|_{L^2(\R^d)}\lesssim e^{Ct},\quad t\ge 0. 
  \end{equation*}
\end{lemma}
We infer from \eqref{ineq:GN} and \eqref{eq:GNdual}  that if $u_0\in \Sigma$, for $0<\eta<\min
\(1,\frac{4}{d+2}\)$, there exists $C>0$
such that,
\begin{equation}
  \label{eq:varphiLp}
  \|\varphi_{\rm app}^\eps(t)\|_{L^{2-\eta}}^{2-\eta}\lesssim
  \eps^{\frac{d\eta}{4}}e^{Ct}\quad;\quad 
  \|\varphi_{\rm app}^\eps(t)\|_{L^{2+\eta}}^{2+\eta}\lesssim
  \eps^{-\frac{d\eta}{4}}e^{Ct},\quad t\ge 0. 
\end{equation}
We now show that the exact linear solution $\varphi^\eps$ satisfies
similar estimates. Define $v^\eps$ by the same rescaling as the
one relating $\varphi_{\rm app}^\eps$ and $v$,
\begin{equation*}
  \varphi^\varepsilon(t,x)=\frac{1}{\varepsilon^{d /
        4}}v^\eps\(t,\frac{x-q(t)}{\sqrt{\varepsilon}}\) e^{i\phi_{\rm
        lin}(t, x)/\eps}. 
  \end{equation*}
Direct
  computations (see e.g. \cite{Carles2021book}) show that $v^\eps$ solves
\begin{equation}\label{eq:veps}
i \partial_tv^{\varepsilon}+\frac{1}{2} \Delta v^{\varepsilon}=V^{\varepsilon}(t,y)v^{\varepsilon},
\end{equation}
where
\begin{equation*}
V^\eps(t,y) = \frac{1}{\eps}\(V\(q(t)+y\sqrt \eps\) -
V\(q(t)\)-\sqrt\eps y\cdot\nabla V \(q(t)\)\) . 
\end{equation*}
We classically have $\|v^\eps(t)\|_{L^2(\R^d)}=\|u_0\|_{L^2(\R^d)}$
for all $t\in \R$. 
Taylor formula yields
\begin{equation*}
 V^\eps(t,y) = \int_0^1(1-\theta)\<y,\nabla^2 V\(q(t)+\theta
 y\sqrt\eps\)y\>d\theta.  
\end{equation*}
In view of Assumption~\ref{hyp:V}, there exists $C>0$ independent of
$\eps\in (0,1)$ and $t\in \R$ such that
\begin{equation*}
  |\nabla V^\eps(t,y)|\le C\(1+|y|\). 
\end{equation*}
Applying the operator $\nabla$ to \eqref{eq:veps}, we get:
\begin{align*}
\left(i \partial_t+\frac{1}{2} \Delta\right) \nabla
  v^\varepsilon=
  V^\eps(t,y)\nabla v^\eps + v^\eps \nabla V^\eps(t,y),
\end{align*}
and the energy estimate yields
\begin{align*}
  \|\nabla  v^\varepsilon(t)\|_{L^2}
   &\lesssim \|\nabla u_0\|_{L^2}+
     \int_0^t \|v^\varepsilon(s)\nabla V^\eps(s)\|_{L^2}ds \\
  &\lesssim
    \|\nabla u_0\|_{L^2}+ 
  \int_0^t \|(1+|y|)v^\varepsilon(s)\|_{L^2}ds .
\end{align*}
Applying the multiplication operator $y$ to \eqref{eq:error&alpha>1},
we get similarly: 
\begin{equation*}
\left(i\partial_t+\frac{1}{2} \Delta\right) \(y
v^\varepsilon\)=V^{\varepsilon}y
v^\varepsilon+\nabla v^\varepsilon,
\end{equation*}
and, by energy estimate,
\begin{align*}
  \|y  v^\varepsilon(t)\|_{L^2}
  & \lesssim \|y u_0\|_{L^2}+
  \int_0^t \|\nabla v^\varepsilon(s)\|_{L^2}ds .
\end{align*}
Summing these two inequalities, Gr\"onwall lemma yields
\begin{equation}\label{eq:controlveps}
   \|\nabla  v^\varepsilon(t)\|_{L^2}+ \|y
  v^\varepsilon(t)\|_{L^2} \lesssim \|u_0\|_{\Sigma}e^{Ct},\quad t\ge
  0. 
\end{equation}
Resuming the computations presented in the case of $\varphi_{\rm
  app}^\eps$, we infer:
\begin{lemma}\label{lem:varphiLp}
Let  $d\ge 1$, $0<\eta<\min
\(1,\frac{4}{d+2}\)$, and $u_0\in \Sigma$. There exists $C>0$
independent of $\eps\in 
(0,1)$ and $t\ge 0$ such that
\begin{equation*}
  \|\varphi^\eps(t)\|_{L^{2-\eta}}^{2-\eta}\lesssim
  \eps^{\frac{d\eta}{4}}e^{Ct}\quad;\quad 
  \|\varphi^\eps(t)\|_{L^{2+\eta}}^{2+\eta}\lesssim
  \eps^{-\frac{d\eta}{4}}e^{Ct},\quad t\ge 0. 
\end{equation*}
\end{lemma}

\subsection{Linear approximation in the subcritical case}
In this subsection, we prove the first part of
Proposition~\ref{prop:subcritical}. We assume $u_0\in \Sigma$, and set 
$w^\varepsilon=\psi^\varepsilon-\varphi^\varepsilon$. This error satisfies 
\begin{align*}
i\varepsilon\partial_t w^\varepsilon+\frac{\varepsilon^2}{2} \Delta w^\varepsilon=& V(x) w^\varepsilon +\lambda\varepsilon^\alpha \left(\psi^\varepsilon\log|\psi^\varepsilon|^2-\varphi^\varepsilon\log|\varphi^\varepsilon|^2\right)\\
&+\lambda\varepsilon^\alpha\varphi^\varepsilon\log|\varphi^\varepsilon|^2\quad
  ;\quad w^\eps_{\mid t=0}=0. 
\end{align*}
Recall the strategy for the energy estimate: multiply both sides by
$\overline{w^\eps}$, integrate, and take the imaginary part,  
\begin{align*}
    \frac{\varepsilon}{2} \frac{d}{d
  t}\|w^\varepsilon\|_{L^2}^2
  &=\lambda \varepsilon^\alpha \IM \int
    \overline{\psi^\varepsilon-\varphi^\varepsilon}\left(\psi^\varepsilon
   \log |\psi^\varepsilon|^2-\varphi^\varepsilon
    \log  |\varphi^\varepsilon|^2\right)+\lambda\varepsilon^\alpha\IM
     \int  \overline{w^\varepsilon}\varphi^\varepsilon\log|\varphi^\varepsilon|^2\\ 
 &\le 2|\lambda| \varepsilon^\alpha
   \|w^\varepsilon\|_{L^2}^2+|\lambda| \varepsilon^\alpha
   \|w^\varepsilon\|_{L^2}\|\varphi^\varepsilon \log
   |\varphi^\varepsilon|^2\|_{L^2}, 
\end{align*}
where we have used Lemma~\ref{lem:CH} and Cauchy-Schwarz inequality. 
Gr\"onwall lemma yields 
\begin{equation}\label{eq:gronwallsub1}
    \|w^\varepsilon\|_{L^2}\leq\varepsilon^{\alpha-1}\int_0^t e^{C\varepsilon^{\alpha-1}(t-s)}\|\varphi^\varepsilon \log |\varphi^\varepsilon|^2(s)\|_{L^2} ds.
\end{equation}
We estimate  $\varphi^\varepsilon (t,\cdot)\log
|\varphi^\varepsilon(t,\cdot)|^2$ 
in  $ L^2\left(\mathbb{R}^d\)$: for $0<\eta<\min (1,\frac{4}{d+2})$
arbitrarily small, 
\begin{equation}\label{ineq:log to lp}
    |\varphi^\varepsilon|^2\left(\log
      |\varphi^\varepsilon|^2\right)^2
    \lesssim |\varphi^\varepsilon|^{2-\eta}+
    |\varphi^\varepsilon|^{2+\eta},
\end{equation}
and so, in view Lemma~\ref{lem:varphiLp}, 
\begin{equation*}
  \|\varphi^\varepsilon (t)\log
  |\varphi^\varepsilon(t)|^2\|_{L^2}^2\lesssim
  \eps^{-\frac{d\eta}{4}}e^{Ct},\quad t\ge 0. 
\end{equation*}
Plugging this estimate into \eqref{eq:gronwallsub1} yields the first
part of  Proposition~\ref{prop:subcritical}.
\subsection{Error estimate in the linear case}

We prove the end of Proposition~\ref{prop:subcritical}.
Recall that $\varphi^\eps$ solves the linear equation
\eqref{eq:evolution-linear}. Its standard approximation is
$\varphi_{\rm app}^\eps$, see
e.g. \cite{robert2021coherent} or 
\cite{Carles2021book}. 
Denote the error for the envelope by
$  \delta_v^\varepsilon=v^\varepsilon-v$.
It satisfies
\begin{equation}\label{eq:error&alpha>1}
i \partial_t\delta_v^\varepsilon+\frac{1}{2}
\Delta\delta_v^\varepsilon=V^{\varepsilon}(t,y)
\delta_v^\varepsilon+\left(V^{\varepsilon}(t,y)-\frac{1}{2}\left\langle
    y, \nabla^2 V(q(t)) y\right\rangle\right) v. 
\end{equation}
Taylor expansion yields the uniform in time pointwise estimate
\begin{equation*}
  \left| V^{\varepsilon}(t,y)-\frac{1}{2}\left\langle
    y, \nabla^2 V(q(t)) y\right\rangle\right|\lesssim \sqrt\eps |y|^3,
\end{equation*}
hence, by energy estimate, since $\delta_v^\eps=0$ at $t=0$,
\begin{equation}\label{eq:deltavL2}
  \|\delta_v^\eps(t)\|_{L^2} \lesssim \sqrt\eps \int_0^t \||y|^3
  v(s)\|_{L^2}ds\lesssim \sqrt\eps e^{Ct},
\end{equation}
where we have invoked Lemma~\ref{lem:(exp)_k} with $k=3$. 
\smallbreak

The intermediary regularity assumption $u_0\in \Sigma^2$ will be
useful in the critical case $\alpha=1$, and consists in resuming the
energy estimate: multiply \eqref{eq:error&alpha>1} by
$\overline{\delta_v^\eps}$, integrate in space, and take the
imaginary part: Cauchy-Schwarz inequality yields
\begin{equation*}
  \frac{1}{2}\frac{d}{dt}\|\delta_v^\eps\|_{L^2}^2 \lesssim \sqrt\eps
  \||y| \delta_v^\eps\|_{L^2}\||y|^2 v\|_{L^2}.
\end{equation*}
For $u_0\in \Sigma^2$, Lemma~\ref{lem:(exp)_k} guarantees that
\begin{equation*}
  ||y|^2 v(t)\|_{L^2}\lesssim e^{Ct},\quad t\ge 0.
\end{equation*}
On the other hand, triangle inequality yields, for any $t\ge 0$,
\begin{equation*}
  \||y| \delta_v^\eps(t)\|_{L^2}\le \||y|v^\eps(t)\|_{L^2}
  +\||y|v(t)\|_{L^2}\lesssim e^{Ct}, 
\end{equation*}
where we have used \eqref{eq:controlveps} too. We infer
\begin{equation*}
  \frac{1}{2}\frac{d}{dt}\|\delta_v^\eps\|_{L^2}^2 \lesssim \sqrt\eps
  e^{Ct}, 
\end{equation*}
hence the announced error estimate.

\section{Critical case: single coherent state}
\label{sec:critical}

This section is dedicated to the proof of
Theorem~\ref{theo:single}. From now on, and in the next section as
well, we assume $\alpha=1$.

The approximate solution is 
\begin{equation*}
\psi_{\rm app}^{\varepsilon}(t, x)=\frac{1}{\varepsilon^{d / 4}} u\left(t, \frac{x-q(t)}{\sqrt{\varepsilon}}\right) e^{i(S(t)+p(t) \cdot(x-q(t))) / \varepsilon-i\lambda\frac{d}{2}t\log\varepsilon}.
\end{equation*}
Here, as before, $(q, p)$ is given by the Hamiltonian flow
\eqref{eq:hamilton}, the classical action is given by
\eqref{eq:action}, but now the envelope $u$ satisfies \eqref{eq:u}. 
\subsection{General case}
\label{sec:general}
In this subsection, we assume $u_0\in \Sigma^2$. 
The Cauchy problem for \eqref{eq:u} was addressed in
Lemma~\ref{lem:u}, complemented by Proposition~\ref{prop:Sigma2}.
With these preliminaries at hand, we can follow essentially the same
strategy as in Section~\ref{sec:subcritical}. We denote by $w^\eps =
\psi^\eps-\psi_{\rm app}^\eps$ the error. It solves
\begin{align*}
  i\eps\d_t w^\eps +\frac{\eps^2}{2}\Delta w^\eps
  &= V(x)w^\eps
    +\lambda \eps \(\psi^\eps\log|\psi^\eps|^2 -
    \psi_{\rm  app}^\eps\log|\psi_{\rm app}^\eps|^2 \)\\
  &\quad + \(V(x) - T^2_{q(t)}\)
  \psi_{\rm app}^\eps \quad;\quad w^\eps_{\mid t=0}=0,
\end{align*}
where $T^2_{q(t)}$ denote the Taylor expansion of $V$ at order two about
$q(t)$, 
\begin{equation*}
  T^2_{q(t)}(x)=V\(q(t)\)+\(x-q(t)\) \cdot \nabla
  V\(q(t)\)+\frac{1}{2}\left\langle x-q(t), \nabla^2 V\(x-q(t)\)
    \(x-q(t)\)\right\rangle.
\end{equation*}
Multiply the equation in $w^\eps$ by $\overline{w^\eps}$, integrate in
space, and take the imaginary part: using Lemma~\ref{lem:CH} to
estimate the logarithmic term, we find
\begin{equation*}
  \eps \frac{d}{dt}\|w^\eps\|_{L^2}^2 \lesssim \eps \|w^\eps\|_{L^2}^2
  +\int_{\R^d} \left| V(x) - T^2_{q(t)}\right|
|w^\eps(t,x)|  |\psi_{\rm app}^\eps(t,x)| dx.
\end{equation*}
Taylor formula yields the pointwise estimate
\begin{equation*}
  \left| V(x) - T^2_{q(t)}\right|\lesssim \left|x-q(t)\right|^3,
\end{equation*}
and we balance these three powers like in the previous subsection:
\begin{align*}
 \int_{\R^d} \left| V(x) - T^2_{q(t)}\right|
  |w^\eps(t,x)|
  &|\psi_{\rm app}^\eps(t,x)| dx\\
  &\lesssim \left\| |x-q(t)|
  w^\eps(t)\right\|_{L^2}  \left\| |x-q(t)|^2
  \psi_{\rm app}^\eps(t)\right\|_{L^2}  .
\end{align*}
Introduce the exact envelope $u^\eps$ defined by
\begin{equation*}
  \psi^\eps(t,x) = \frac{1}{\varepsilon^{d / 4}} u^\eps
  \left(t, \frac{x-q(t)}{\sqrt{\varepsilon}}\right)
  e^{i(S(t)+p(t) \cdot(x-q(t))) / \varepsilon-i\lambda\frac{d}{2}t\log\varepsilon}.
\end{equation*}
We first remark that $\|u^\eps(t)\|_{L^2}=
\|\psi^\eps(t)\|_{L^2}=\|u_0\|_{L^2}$ for all $t\in \R$. 
We readily check that the arguments presented in
Sections~\ref{sec:prelim} and \ref{sec:subcritical} show that there
exists $C>0$ independent of $\eps$ and $t$ such that, like in
Lemma~\ref{lem:u}, 
\begin{equation*}
  \|\nabla u^\eps(t)\|_{L^2} + \|y u^\eps(t)\|_{L^2} \lesssim
  e^{Ct},\quad t\ge 0.
\end{equation*}
In view of the triangle inequality and Lemma~\ref{lem:u}, this yields
\begin{equation*}
  \left\| |x-q(t)|
    w^\eps(t)\right\|_{L^2} \lesssim \sqrt\eps e^{Ct},\quad t\ge 0.
\end{equation*}
Invoking now Proposition~\ref{prop:Sigma2}, we find
\begin{equation*}
  \left\| |x-q(t)|^2
  \psi_{\rm app}^\eps(t)\right\|_{L^2}  \lesssim \eps e^{e^{Ct}},\quad
t \ge 0,
\end{equation*}
and thus
\begin{equation*}
  \eps \frac{d}{dt}\|w^\eps\|_{L^2}^2 \lesssim \eps \|w^\eps\|_{L^2}^2
  +\eps^{3/2} e^{e^{Ct}}.
\end{equation*}
The first part of Theorem~\ref{theo:single} then follows from
Gr\"onwall lemma. 

\subsection{Gaussian case}
\label{sec:gaussian}
The main remark is that the potential in \eqref{eq:u} being quadratic
in $y$, Gaussian initial data lead to a solution $u$ which is Gaussian
at all time. Solving the partial differential equation \eqref{eq:u}
thus reduces to solving ordinary differential equations that describe
the time-dependent coefficients of the Gaussian $u(t, \cdot)$. We
present computations in the one-dimensional case for
clarity. If $V$ decouples variables, we can use the tensorization
property to address higher dimension.  

\subsubsection{Approximate solution in the Gaussian case}

Suppose $d=1$, and
\begin{equation*}
  u_0(y) = b_0 e^{-a_0y^2/2}, \quad a_0,b_0\in \C, \ \RE a_0>0.
\end{equation*}
The solution to \eqref{eq:u} is sought under the form 
\begin{equation*}
u(t, y)=b(t) e^{-a(t)y^2 / 2}.
\end{equation*}
We compute
\begin{equation*}
\begin{aligned}
& i \partial_t u+\frac{1}{2} \d_y^2 u-\frac{1}{2} V^{\prime\prime}(q(t)) y^2u-\left.\lambda u \log |u|^2\right|_{u=b(t) e^{-a(t)y^2/2}} \\
= & \left(A_1(t)+A_2(t) y^2\right)e^{-a(t) y^2 / 2} .
\end{aligned}
\end{equation*}
where
\begin{equation*}
\begin{aligned}
&A_1(t)=i \dot{b}-\frac{1}{2} a b- \lambda b \log |b|^2,\\
&A_2(t)=-\frac{1}{2} i \dot{a} b+\frac{1}{2} a^2 b-\frac{1}{2}
  V^{\prime \prime} \(q(t)\)b+\lambda  b \RE a.
\end{aligned}
\end{equation*}
We thus cancel $A_1$ and $A_2$,
\begin{equation}\label{eq:ode.b}
    i \dot{b}=\frac{1}{2} a b+ \lambda b \log |b|^2, 
\end{equation}
\begin{equation}\label{eq:ode.a}
    i \dot{a}=a^2-V^{\prime \prime}\(q(t)\) +2\lambda \RE a.
\end{equation}
Equation~\eqref{eq:ode.b} is integrated like in \cite{CaGa18}. Multiplying both sides of \eqref{eq:ode.b} by $\bar b$, 
\begin{equation}\label{eq:ode.b.1}
i \dot{b} \bar{b}=\frac{1}{2} a|b|^2+\lambda|b|^2 \log|b|^2.
\end{equation}
Taking the imaginary part, we find
\begin{equation*}
    \frac{d}{d t}|b|^2=|b|^2 \IM a ,
\end{equation*}
hence
\begin{equation*}
|b(t)|^2 =\left|b_0\right|^2 \exp \left(\IM
  A(t)\right),\quad\text{where }
A(t):=\int_0^t a(s) d s. 
\end{equation*}
Plugging this expression into \eqref{eq:ode.b}, we have
\begin{equation*}
i \dot{b}=\frac{1}{2} a b+\lambda b \log \left|b_0\right|^2+\lambda b \IM A(t),
\end{equation*}
which we directly integrate,
\begin{equation}\label{eq:b}
b(t)=b_0 \exp \left(-i \lambda t\log \left|b_0\right|^2-\frac{i}{2} A(t)-i \lambda \IM \int_0^t A(s) d s\right).
\end{equation}
Like in \cite{CaGa18}, we seek $a$ under the form
\begin{equation}\label{eq:atau}
a(t)=\frac{\alpha_0}{\tau(t)^2}-i \frac{\dot{\tau}(t)}{\tau(t)},
\quad\text{with }
\RE a_0=\alpha_0.
\end{equation}
If $\beta_0=\IM a_0$, we check that \eqref{eq:ode.a} becomes
\begin{equation*}
  \ddot{\tau}=\frac{\alpha_0^2}{\tau^3}+2 \lambda \frac{\alpha_0}{\tau}
  -V^{\prime\prime}\(q(t)\)\tau\quad ; \quad \tau(0)=1, \quad
  \dot{\tau}(0)=-\beta_0 . 
\end{equation*}
Since $V''$ is bounded, the following lemma applies to this $\tau$:
\begin{lemma}\label{lem:bddoftau}
  Let $\lambda\in \R$ and $\Omega\in L^\infty(\R;\R)$. For any
  $\tau_0>0$ and $\tau_1\in 
  \R$, consider the ordinary differential equation
  \begin{equation}
  \label{eq:ODE}
  \ddot\tau = \frac{2\lambda}{\tau}+\frac{1}{\tau^3}
  -\Omega (t)\tau\quad;\quad \tau(0)=\tau_0,\quad \dot
  \tau(0)=\tau_1.
\end{equation}
It has a unique, global solution $\tau\in C^1(\R;\R)$, and it  obeys
the exponential bound 
  \begin{equation*}
    \exists C>0,\quad
    \frac{1}{\tau(t)^2}+\tau(t)^2+\dot\tau(t)^2\lesssim e^{Ct},\quad 
    \forall t\ge 0. 
  \end{equation*}
\end{lemma}

\begin{proof}
 Local existence follows from the Cauchy-Lipschitz Theorem, and by
  continuity, $\tau>0$ at least near $t=0$. By
  symmetry, we now focus on positive time, $t>0$. Since the
 right hand side of \eqref{eq:ODE}  is smooth away from the origin,
 either the solution is global, or there exists $T>0$ such that
 \begin{equation*}
   \frac{1}{\tau(t)}+\tau(t)+|\dot\tau(t)| \Tend t T +\infty. 
 \end{equation*}
(The presence of $\dot \tau$ follows when turning \eqref{eq:ODE} into
a first order equation, to apply the Cauchy-Lipschitz Theorem.)
We first show that as long as $\tau$ is bounded, so are $1/\tau$ and
$|\dot \tau|$. We
then prove the a priori bound on $\tau$ stated in the lemma, which implies that the
solution is global.  
\smallbreak

Multiply \eqref{eq:ODE} by $\dot \tau$, and integrate between $0$ and $t$:
\begin{equation}
  \label{eq:int}
  \dot \tau(t)^2 = \underbrace{\tau_1^2+4\lambda \ln \tau_0
    +\frac{1}{\tau_0^2}}_{=:C_0} 
  +4\lambda \ln\tau(t)-\frac{1}{\tau(t)^2}-2\int_0^t \Omega(s)\tau(s)\dot\tau(s)ds.
\end{equation}
For any $\eps>0$ there exists $C_\eps>0$ such that
\begin{equation}\label{eq:eps}
  |\ln \tau|\le C_\eps + \frac{\eps}{\tau^2}+\eps\tau^2,\quad \forall
  \tau>0. 
\end{equation}
For $t$ bounded and $\eta>0$, estimate the integral by
\begin{align*}
  2\left| \int_0^t \Omega(s)\tau(s)\dot\tau(s)ds\right| &\le
  2t\|\Omega\|_{L^\infty}\sup_{s\in [0,t]} |\tau(s)||\dot \tau(s)|\\
&\le t\|\Omega\|_{L^\infty}\sup_{s\in [0,t]}
   \(\frac{1}{\eta}|\tau(s)|^2+\eta|\dot \tau(s)|^2\)\\
&\le t\frac{\|\Omega\|_{L^\infty}}{\eta}\sup_{s\in [0,t]}
  |\tau(s)|^2+t\eta\|\Omega\|_{L^\infty}\sup_{s\in [0,t]}|\dot \tau(s)|^2,
\end{align*}
where we have used Young inequality. Taking $\eps$ and $\eta$
sufficiently small, we infer
\begin{equation}\label{eq:1/tau}
  \sup_{s\in [0,t]}\(\dot\tau(s)^2+\frac{1}{\tau(s)^2}\)
  \le C(t)\( 1 + \sup_{s\in [0,t]}  |\tau(s)|^2\).
\end{equation}
In particular, if $\tau$ is bounded on $[0,T]$, so is $1/\tau$. 
Introduce
\begin{equation*}
  y(t) = \tau(t)^2+\dot \tau(t)^2\ge 0.
\end{equation*}
We compute 
\begin{equation*}
  \dot y = 2\tau\dot\tau+2\dot\tau\ddot\tau=2\tau\dot\tau
  +2\lambda\frac{\dot\tau}{\tau} + \frac{\dot\tau}{\tau^3}- \Omega\tau\dot\tau. 
\end{equation*}
Integrating between $0$ and $t$, and using Young inequality (for the
term $\int_0^t \tau\dot\tau$), 
\begin{equation*}
  y(t) \le C + \int_0^t y(s)ds + 2\lambda
  \ln\tau(t)-\frac{1}{2\tau(t)^2}-\int_0^t \Omega(s)\tau(s)\dot\tau(s)ds. 
\end{equation*}
Using \eqref{eq:eps} again, and Young inequality in the last integral,
\begin{equation*}
   y(t) \le C + \int_0^t y(s)ds + 2|\lambda|\( C_\eta +
   \frac{\eta}{\tau(t)^2}+\eta\tau(t)^2
   \)-\frac{1}{2\tau(t)^2}+\|\Omega\|_{L^\infty}\int_0^t 
  y(s)ds.  
\end{equation*}
For $2|\lambda|\eta< 1/2$, the factor of $1/\tau^2$ on the right
hand side is negative, and  the term
$2|\lambda|\eta \tau^2$ is absorbed by the left hand side, so we get
\begin{equation*}
  y(t) \lesssim 1 + \int_0^t y(s)ds.
\end{equation*}
In view of the above discussion, Gr\"onwall lemma then yields global
existence and the exponential bound for $\tau$ and $\dot \tau$. The
exponential bound for $1/\tau$ then follows from 
\eqref{eq:int} and \eqref{eq:eps}.
\end{proof}
We infer the analogue of Lemma~\ref{lem:(exp)_k} in the nonlinear case
for Gaussian data:
\begin{lemma}\label{lem:gaussianExpk}
  Let $d\ge 1$. Suppose that $V$ decouples variables,
  \begin{equation*}
    V(x) = \sum_{j=1}^d V_j(x_j),
  \end{equation*}
and that $u_0$ is a Gaussian of the form
  \begin{equation*}
    u_0(y) = b_0 \exp \(-\frac{1}{2}\sum_{j=1}^da_{0j}y_j^2\), \quad
    a_{0j},b_0\in \C,\ \RE a_{0j}>0. 
  \end{equation*}
 For any  $k\ge
  1$,  there exists $C_k>0$ such that the solution $u\in L^\infty_{\rm
    loc}(\R;\Sigma^k)$ to \eqref{eq:u} satisfies
  $\|u(t)\|_{L^2}=\|u_0\|_{L^2}$ and, for all $\beta
  \in \N^d$ with $1\le |\beta|\le k$, 
  \begin{equation*}
    \|\d_y^\beta u(t)\|_{L^2(\R^d)}+ \|y^\beta
    u(t)\|_{L^2(\R^d)}\lesssim e^{C_kt},\quad t\ge 0. 
  \end{equation*}
\end{lemma}
\begin{proof}
  First, we note that the above one-dimensional computation is readily
  adapted, by seeking $u(t,y)$ under the form
  \begin{equation*}
    u(t,y) = b(t) \exp \(-\frac{1}{2}\sum_{j=1}^da_{j}(t)y_j^2\).
  \end{equation*}
  Equation~\eqref{eq:ode.b} becomes
  \begin{equation*}
   i \dot{b}=\frac{1}{2} \sum_{j=1}^da_j b+ \lambda b \log |b|^2,  
 \end{equation*}
 and \eqref{eq:ode.a} becomes a family of $d$ uncoupled equations,
 since $\nabla^2V$ is diagonal,
 \begin{equation*}
    i \dot{a_j}=a_j^2-V_j''\(q_j(t)\) +2\lambda \RE a_j,
  \end{equation*}
so we can use directly the one-dimensional analysis. For $\beta\in
\N^d$, we readily compute
\begin{equation*}
  \|y^\beta u(t) \|_{L^2} = |b(t)| \prod_{j=1}^d \|y_j^{\beta_j}
  e^{-a_j(t)y_j^2/2}\|_{L^2} = |b(t)|\prod_{j=1}^d \(\RE
  a_j(t)\)^{-\beta_j/2-1/4} \mu(\beta_j),
\end{equation*}
where the coefficients $\mu(\beta_j)$ stand for the momenta of the
Gaussian in $L^2$,
\begin{equation*}
  \mu(\beta) =\( \int_{\R} z^{2\beta} e^{-z^2}dz\)^{1/2},
\end{equation*}
whose precise value is irrelevant here. As
\begin{align*}
  |b(t)|^2 &= |b_0|^2\exp \(\sum_{j=1}^d \int_0^t \IM a_j(s)ds\)\\
  &=
 |b_0|^2\exp \(-\sum_{j=1}^d \int_0^t \frac{\dot
   \tau_j(s)}{\tau_j(s)}ds\)
 = |b_0|^2\prod_{j=1}^d\tau_j(t)^{-1/2},
\end{align*}
with notations inherited from the one-dimensional case, we have
\begin{equation*}
    \|y^\beta u(t) \|_{L^2} \lesssim \prod_{j=1}^d\tau_j(t)^{-1/2}\(\RE
  a_j(t)\)^{-\beta_j/2-1/4}  \lesssim
  \prod_{j=1}^d\tau_j(t)^{-1/2+\beta_j+1/2}.  
\end{equation*}
Lemma~\ref{lem:bddoftau} yields $0<\tau_j(t)\lesssim e^{c_j t}$ for
some $c_j>0$, hence
the announced bound for the momenta. The bound for the derivative
proceeds similarly, by considering the Fourier transform of
$u(t,\cdot)$. 
\end{proof}
\subsubsection{Error estimate in the Gaussian case}
We can now resume the notations and computations from
Section~\ref{sec:general}. The energy estimate for $w^\eps$ implies
\begin{equation*}
  \|w^\eps(t)\|_{L^2}\lesssim \frac{1}{\eps}\int_0^t e^{C(t-s)} \left\| |x-q(t)|^3
    \psi_{\rm app}^\eps(s)\right\|_{L^2}ds\lesssim \sqrt\eps \int_0^t
  e^{C(t-s)}  \left\| |y|^3  u(s)\right\|_{L^2}ds,
\end{equation*}
so Lemma~\ref{lem:gaussianExpk} completes the proof of
Theorem~\ref{theo:single}.

\section{Superposition in the Gaussian case}
\label{sec:superp}

In this section, we prove Theorem~\ref{theo:superp}. A similar result
was proven in \cite{CaFe11} in the case of a smooth power nonlinearity
($\eps^{\alpha_c}|\psi^\eps|^{2\si}\psi^\eps$, with $\si$ an integer), and
we resume the same strategy here. The major difference however lies in
the estimate of interaction terms. Typically, if $\psi^\eps_{1,\rm
  app}$ and $\psi^\eps_{2,\rm
  app}$ have difference centers in physical space, $q_1\not = q_2$,
the error term
\begin{equation*}
  | \psi^\eps_{1,\rm  app}+\psi^\eps_{2,\rm  app}|^2\(
  \psi^\eps_{1,\rm  app}+\psi^\eps_{2,\rm  app}\) -  
| \psi^\eps_{1,\rm  app}|^2
  \psi^\eps_{1,\rm  app}
- | \psi^\eps_{2,\rm  app}|^2\psi^\eps_{2,\rm  app}
\end{equation*}
involves only cross terms, where the factor $\psi^\eps_{1,\rm
  app}\times\psi^\eps_{2,\rm  app}$ (leaving out the conjugacy to
simplify the discussion) is present: since $\psi^\eps_{j,\rm  app}$ is
centered in $q_j$ at scale $\sqrt\eps$, these cross terms can be shown
to be small. In the case of a logarithmic nonlinearity, we cannot use
the same sort of expansion: instead, we rely on a lemma proven by
Ferriere in \cite{Ferriere2020}, which measures precisely the analogue of this
decomposition, \emph{in the case of Gaussian functions} (see
Lemma~\ref{lem:separation ineq}
below).

We examine the superposition principle in \eqref{eq:logNLS}, and
distinguish two regions. When the
trajectories (in physical space) of the two centers intersect, or are
close to each other, we keep the nature of the nonlinear interaction
as a black box, and use the fact  such an event is sufficiently
rare for the nonlinear interaction to be negligible. This step is an
adaptation of computations from \cite{CaFe11} to the case of a
logarithmic nonlinearity. Conversely, when the centers of the two
coherent states are far from each other, we apply the method 
of Ferriere \cite{Ferriere2020}, which shows the absence of
interaction at leading order. 

\subsection{Preparation of the proof}
\label{sec:prep}
As announced above, the scheme of the proof is the same as in
\cite{CaFe11}, and we resume the same notations as much as possible. 
Under the assumptions of Theorem~\ref{theo:superp}, denote $w^{\varepsilon}=\psi^{\varepsilon}-\psi_{1,\rm app}^\eps
-\psi_{2,\rm app}^\eps$ the error function. It solves
\begin{equation}\label{eq:the error of two centre}
    i \varepsilon \partial_t w^{\varepsilon}+\frac{\varepsilon^2}{2}
    \Delta w^{\varepsilon}=V(x) w^{\varepsilon}+L^{\varepsilon}+
    \mathcal{N}^{\varepsilon}\quad ; \quad w_{\mid
      t=0}^{\varepsilon}=0, 
\end{equation}
where we have now
\begin{equation*}
  L^{\varepsilon}(t, x)=\left(V(x)-T_{q_1(t)}^2(x)\right) \psi_{1,\rm app}^{\varepsilon}(t, x)+\left(V(x)-T_{q_2(t)}^2(x)\right) \psi_{2,\rm app}^\eps(t, x),
\end{equation*}
and, like in Section~\ref{sec:critical}, $T_z^2$ stands for the Taylor
second order polynomial of $V$ about $z$, 
\begin{equation*}
  T^2_{z}(x)=V(z)+(x-z) \cdot \nabla V(z)+
  \frac{1}{2}\left\langle x-z, \nabla^2 V(z)(x-z)\right\rangle.
\end{equation*}
The approximate functions associated with each Gaussian are given by
\begin{equation*}
  \psi_{j,\rm app}^\eps(t,x) = \frac{1}{\varepsilon^{d / 4}}
  u_j\left(t, \frac{x-q_j(t)}{\sqrt{\varepsilon}}\right) e^{i(S_j(t)+p_j(t)
    \cdot(x-q_j(t))) /
    \varepsilon-i\lambda\frac{d}{2}t\log\varepsilon},
\end{equation*}
where $q_j$, $p_j$ and $S_j$ are defined by \eqref{eq:hamilton} (with
initial data $q_{0j}$ and $p_{0j}$) and \eqref{eq:action}, while $u_j$
solves \eqref{eq:u} with initial datum $u_{0j}$. 
Denote by $g$ the function
\begin{equation*}
  g(z) = \lambda z\log|z|^2, \quad z\in \C.
\end{equation*}
The nonlinear source term is given by
\begin{equation*}
  \mathcal{N}^{\varepsilon}
  =\varepsilon\(g\(w^{\varepsilon}+\psi_{1,\rm
    app}^\eps+\psi_{2,\rm app}^\eps\) -g\(\psi_{1,\rm app}^\eps\)-
  g\(\psi_{2,\rm app}^\eps\)\).
\end{equation*}
We decompose $\mathcal{N}^{\varepsilon}$ as the sum of a semilinear
term and an interaction source term:
$\mathcal{N}^{\varepsilon}=\mathcal{N}_S^{\varepsilon}+\mathcal{N}_I^{\varepsilon}$,
where 
\begin{equation*}
\begin{aligned}
& \mathcal{N}_S^{\varepsilon}=\varepsilon\left(g(w^{\varepsilon}+\psi_{1,\rm app}^\eps+\psi_{2,\rm app}^\eps)-g(\psi_{1,\rm app}^\eps+\psi_{2,\rm app}^\eps)\right), \\
& \mathcal{N}_I^{\varepsilon}=\varepsilon\left(g(\psi_{1,\rm app}^\eps+\psi_{2,\rm app}^\eps)-g\(\psi_{1,\rm app}^\eps\)-g\(\psi_{2,\rm app}^\eps\)\right).
\end{aligned}
\end{equation*}
Multiply \eqref{eq:the error of two centre} by $\overline {w^\varepsilon}$,
integrate over $\mathbb{R}^d$, and take the imaginary part: we get
\begin{equation*}
\frac{\eps}{2} \frac{d}{d t}\|w^\varepsilon\|_{L^2}^2
\le\|w^\varepsilon\|_{L^2}
\left(\|L^\varepsilon\|_{L^2}
  +\left\|\mathcal{N}_I^{\varepsilon}\right\|_{L^2}\right)+\operatorname{Im}
\int_{\R^d} \overline{w^\eps}\mathcal{N}_S^{\varepsilon}. 
\end{equation*}
The linear
source term $L^{\varepsilon}$ is handled like in the case of a single
coherent state (Section~\ref{sec:critical}):
\begin{equation*}
  \|L^\varepsilon(t)\|_{L^2}\lesssim \sum_{j=1,2}
  \left\||x-q_j(t)|^3\psi_{j,\rm app}^\eps(t)\right\|_{L^2} = \eps^{3/2} \sum_{j=1,2}
  \left\||y|^3u_j(t)\right\|_{L^2}
\end{equation*}
In view of Lemma~\ref{lem:gaussianExpk}, we infer
\begin{equation*}
    \|L^\varepsilon(t)\|_{L^2}\lesssim \eps^{3/2} e^{Ct},\quad t\ge 0.
\end{equation*}
Also, in view of Lemma~\ref{lem:CH},
\begin{equation*}
  \operatorname{Im}
\int_{\R^d} \overline{w^\eps} \mathcal{N}_S^{\varepsilon}\le
2|\lambda| \|w^\eps\|_{L^2}^2,
\end{equation*}
so this term will be handled by Gr\"onwall lemma. 
Therefore the core of the proof lies in the control of
$N_I^\varepsilon$. Like in \cite{CaFe11}, we separate the time
interval that we consider into two parts: when the trajectories $q_1$
an $q_2$ are close to each other, and when they are not, each case
being treated with different arguments.

\subsection{Crossing of coherent states}

The analysis of the set where trajectories are close to each other was
performed in \cite{CaFe11}:
\begin{proposition}[Lemma~ 6.2 and Proposition~6.3 in
  \cite{CaFe11}]\label{prop:crossing} 
 For $j=1,2$, let $(q_j,p_j)$ be the
 solutions to \eqref{eq:hamilton} associated with initial data
 $(q_{01},p_{01})\not = (q_{02},p_{02})$. Let $0<\gamma<1/2$. For
 $T>0$, set 
    \begin{equation*}
     I^{\varepsilon}(T)=\left\{t \in[0, T],\left|q_1(t)-q_2(t)\right| \le \varepsilon^\gamma\right\}. 
    \end{equation*}
$\bullet$ If $T$ is  independent of
    $\varepsilon$, we have
\begin{equation*}
    \left|I^{\varepsilon}(T)\right|=\mathcal{O}\left(\varepsilon^\gamma\right).
\end{equation*}
$\bullet$ If in addition $d=1$ and $E_1\not =E_2$, where
\begin{equation*}
  E_j = \frac{p_{0j}^2}{2}+V\(q_{0j}\),
\end{equation*}
then there exist $C,C_0>0$ independent of
$\eps\in (0,1)$ such that
\begin{equation*}
   |I_\eps(t)|\lesssim \eps^\gamma e^{C_0t} |E_1-E_2|^{-2},\quad
  0\le t\le C\log\frac{1}{\eps}.
\end{equation*}
\end{proposition}

In the region $I^\eps$, nonlinear interaction may occur; however,
Proposition~\ref{prop:crossing} implies that the effects of such
interaction remain negligible at leading order. Indeed, we argue like
in Section~\ref{sec:subcritical}, and write,  for
$0<\eta<\min\(1,\frac{4}{d+2}\)$, using \eqref{ineq:log
  to lp},
\begin{equation*}
    \|\mathcal{N}_I^{\varepsilon}\|_{L^2(\mathbb{R}^2)}^2\lesssim
    \varepsilon^2\sum_{j=1,2}\left(\|\psi_{j,\rm
        app}^\varepsilon\|_{L^{2+\eta}}^{2+\eta}+
\|\psi_{j,\rm app}^\varepsilon\|_{L^{2-\eta}}^{2-\eta}\right).  
\end{equation*}
Like in Section~\ref{sec:subcritical}, and now referring to
Lemma~\ref{lem:gaussianExpk} in the nonlinear case for Gaussian
profiles, this implies
\begin{equation*}
    \|\mathcal{N}_I^{\varepsilon}(t)\|_{L^2(\mathbb{R}^2)}\lesssim
    \eps^{1-\frac{d\eta}{8}}e^{Ct},\quad t\ge 0.  
\end{equation*}
Proposition~\ref{prop:crossing} then implies
\begin{equation}\label{eq:source-crossing}
\int_{I^{\varepsilon}(T)}\left\|\mathcal{N}_I^{\varepsilon}(t)\right\|_{L^2}
d t \lesssim  \int_{I^\varepsilon(T)}\eps^{1-\frac{d\eta}{8}}e^{Ct}dt
\lesssim\eps^{\gamma+1-\frac{d\eta}{8}}e^{CT}.
\end{equation}
We emphasize the fact that in view of Proposition~\ref{prop:crossing},
the above estimate is valid in general for $T>0$ independent of
$\eps\in (0,1)$, and, when $d=1$, for $T\le C\log\frac{1}{\eps}$ for
some $C>0$ independent of $\eps$. 
\subsection{Separation of coherent states}

The  complement of the set
$I^\varepsilon(T)$  is the set of times is when the trajectories of
two coherent states are sufficiently far 
apart.  We recall a crucial result from
\cite{Ferriere2020}, which shows some orthogonality phenomenon in the
case of Gaussian functions for the logarithmic nonlinearity. We state
a simplified version of the result from \cite{Ferriere2020}, in the
sense that the assumptions considered there are more general. 
\begin{lemma}[From Lemma~3.2  in \cite{Ferriere2020}]\label{lem:separation ineq}
    For any $d \ge 1$, there exists $C_d>0$ such that the following
    holds. Let $N \in \mathbb{N}^*$ and consider, for $k=1,
    \ldots, N$, $q_k \in \mathbb{R}^d$,
    $\omega_k \in \mathbb{R}$, $\lambda_k>0$ and $\theta_k: \mathbb{R}^d
    \to\mathbb{R}$ a real measurable function. Define, for all $x \in
    \mathbb{R}^d$,   
    \begin{equation*}
      \varphi(x)=\sum_{k=1}^N \varphi_k(x) ,\quad \varphi_k(x)=e^{i
        \theta_k(x)+\omega_k-\lambda_k|x-q_k|^2}.
    \end{equation*}
If
\begin{equation*}
  \rho:=\left(\min _{k \neq
      j}\left|q_j-q_k\right|\right)^{-1}<\rho_0:=\min 
  \left(\frac{\sqrt{\lambda_+}}{\max (\sqrt{\delta \omega+1},
      \sqrt{\ln N})}, \sqrt{\frac{\lambda_-}{d+2}}\right) ,
\end{equation*}
where $\dis\lambda_==\max_{k}\lambda_k$, $\dis\lambda_-=\min_k\lambda_k$, and $\dis\delta \omega:=\max_{j, k}\left|\omega_k-\omega_j\right|$, then
\begin{equation*}
\left\|g\(\sum_{k=1}^N \varphi_k \)-\sum_{k=1}^N g\(\varphi_k \)\right\|_{L^2\left(\mathbb{R}^d\right)} \le C_d N^{\frac{3}{2}} \frac{\lambda_+}{\rho^{\frac{d}{2}+1} \sqrt{\lambda_-}} e^{-\frac{\lambda_-}{4 \rho^2}+\max _j \omega_j} .
\end{equation*}
\end{lemma}
We apply this lemma with $N=2$, 
\begin{align*}
&\theta_k(t,x)=\arg b_k(t)+\frac{\phi_{\rm nl}(t,x)}{\eps}-\IM
      a_k(t)\frac{|x-q_k(t)|^2}{2\varepsilon},\\
&  \omega_k(t)=\log
  |b_k(t)|-\frac{d}{4}\log\varepsilon,\quad
\lambda_k(t)=\frac{\RE a_k(t)}{2\varepsilon}.
\end{align*}
For $t\in [0,T]\setminus{I^\varepsilon(T)}$, we have
\begin{equation*}
\(e^{-Ct}\lesssim\)  \rho\le \eps^{-\gamma}.
\end{equation*}
On the other hand, in view of \eqref{eq:b} and \eqref{eq:atau},
Lemma~\ref{lem:gaussianExpk} yields  $\lambda_k(t)\gtrsim
\eps^{-1}e^{-Ct} $ and  
\begin{equation*}
  \rho_0\gtrsim \frac{1}{\sqrt\eps}e^{-Ct},
\end{equation*}
so the condition from Lemma~\ref{lem:separation ineq} is fulfilled
provided that
$  e^{Ct}\lesssim \eps^{\gamma-1/2},$ 
which is granted with $0\le t\le c\log\frac{1}{\eps}$ for some
$c>0$ independent of $\eps\in (0,1)$, since $\gamma<1/2$. 
Lemma~\ref{lem:separation ineq}
then implies, for $t\in [0,T]\setminus{I^\varepsilon(T)}$, and
recalling Lemma~\ref{lem:bddoftau} to bound $\RE a_k$ from above,
\begin{equation}\label{eq:source-away}
  \frac{1}{\eps}\|\mathcal N_I^\eps(t)\|_{L^2}\lesssim
  \eps^{-\frac{d}{4}-\frac{1}{2}}
  e^{Ct}\exp\(-C\eps^{2\gamma-1}e^{-Ct}\), \quad t\in [0,T]\setminus
  I^\eps(T). 
\end{equation}

\subsection{Conclusion}
\label{sec:conclusion}

Putting \eqref{eq:source-crossing} and \eqref{eq:source-away}
together, for any $0<\eta<\min\(1,\frac{4}{d+2}\)$, 
\begin{equation*}
  \frac{1}{\eps}\int_0^T \|\mathcal N_I^\eps(t)\|_{L^2}dt \lesssim 
\eps^{\gamma-\frac{d\eta}{8}}e^{CT} + \eps^{-\frac{d}{4}-\frac{1}{2}}
  e^{CT}\exp\(-C\eps^{2\gamma-1}e^{-CT}\).
\end{equation*}
In view of the preliminary estimates from
Section~\ref{sec:prep}, we get
\begin{align*}
  \frac{d}{dt}\|w^\eps(t)\|_{L^2}
&\lesssim \|w^\eps(t)\|_{L^2} +
  \frac{1}{\eps}\|L^\eps(t)\|_{L^2}+
  \frac{1}{\eps}\|\mathcal N_I^\eps(t)\|_{L^2}\\
&\lesssim \|w^\eps(t)\|_{L^2} +
  \sqrt \eps e^{Ct}+
  \frac{1}{\eps}\|\mathcal N_I^\eps(t)\|_{L^2},
\end{align*}
hence from Gr\"onwall lemma and the above estimate, for any
$0<\eta<\min\(1,\frac{4}{d+2}\)$,  
\begin{align*}
 \sup_{t\in [0,T]} \|w^\eps(t)\|_{L^2}
&\lesssim 
  \sqrt\eps e^{CT} +   \frac{1}{\eps}\int_0^T e^{C(T-t)}\|\mathcal
  N_I^\eps(t)\|_{L^2}dt\\
&\lesssim \sqrt\eps e^{CT} +\eps^{\gamma-\frac{d\eta}{8}}e^{CT} +
\eps^{-\frac{d}{4}-\frac{1}{2}} 
  e^{CT}\exp\(-C\eps^{2\gamma-1}e^{-CT}\).
\end{align*}
Recall that this estimate is always true for $T>0$ independent of
$\eps\in (0,1)$, and that if $d=1$, then it holds more generally for
$0<T\le c_1\log\frac{1}{\eps}$ for some $c_1>0$ independent of
$\eps$. Up to decreasing $c_1$, the last term is always
$\O(\eps^{\gamma})$ (this is obvious for $T$ independent of $\eps$). 
This yields the
conclusion of Theorem~\ref{theo:superp}, by replacing $\gamma$ with
$\gamma-\frac{d\eta}{8}$. But since the condition $0<\gamma<1/2$ is
open, and $\eta>0$ is arbitrarily small, $\gamma-\frac{d\eta}{8}$ is
arbitrarily close to $1/2$, hence the result.

\bibliographystyle{abbrv}
\bibliography{references}

\end{document}